\documentclass[11pt]{amsart}
\usepackage[T1]{fontenc}
\usepackage{amsmath}
\usepackage{amsthm}
\usepackage{amsfonts}
\usepackage{amssymb}
\usepackage{graphicx}
\usepackage[english]{babel}
\usepackage[margin=1in]{geometry}
\usepackage{indentfirst}
\usepackage{enumerate}
\usepackage{thm-restate}
\usepackage{subcaption}
\usepackage{xcolor}
\usepackage[percent]{overpic}

\theoremstyle{plain}
\newtheorem{theorem}{Theorem}[section]
\newtheorem{corollary}[theorem]{Corollary}
\newtheorem{lemma}[theorem]{Lemma}
\newtheorem{proposition}[theorem]{Proposition}

\newtheorem*{theorem*}{Theorem}
\newtheorem*{corollary*}{Corollary}

\theoremstyle{definition}
\newtheorem{definition}[theorem]{Definition}

\newtheorem{fact}[theorem]{Fact}

\theoremstyle{remark}
\newtheorem{remark}[theorem]{Remark}
\newtheorem*{remark*}{Remark}

\makeatletter
\@ifundefined{acks}{
    \newenvironment{acks}[1][Acknowledgments]
        {\section*{#1}}
        {}
}{}
\makeatother

\definecolor{linkcol}{RGB}{0,70,160}
\definecolor{citecol}{RGB}{0,120,70}
\definecolor{urlcol}{RGB}{150,60,0}

\usepackage[
    colorlinks=true,
    linkcolor=linkcol,
    citecolor=citecol,
    urlcolor=urlcol,
    breaklinks=true,
    bookmarksopen=true,
    bookmarksnumbered=true
]{hyperref}

\newcommand{\dary}{\mathcal T_h}
\newcommand{\dreg}{\widehat {\mathcal T}_h}
\newcommand{\pc}{p_{\mathsf{c}}}
\newcommand{\dist}{\mathrm{dist}}

\title{The $q<1$ Random-Cluster Model on Wired Trees: Uniqueness and Negative Dependence}
\author{Heehyun Park}
\address{Department of Computer Science and Engineering, Pennsylvania State University}
\email{hbp5148@psu.edu}

\begin{document}

\begin{abstract}
The random-cluster model with cluster weight $0<q<1$ is expected to exhibit negative dependence, but without FKG even pairwise negative correlation remains open on general graphs. We study the model on the infinite $\Delta$-regular tree with wired boundary conditions. Write $\widehat p:=p/(p+q(1-p))$ and $p_{\mathsf c}:=q/(\Delta+q-2)$. Classical results identify the product wired state for $p\le p_{\mathsf c}$ and construct a percolative all-wired limit for $p>p_{\mathsf c}$. We prove that the supercritical wired DLR specification has a unique Gibbs measure, namely this all-wired limit. Consequently, the wired DLR phase diagram is complete: the unique measure is Bernoulli bond percolation with parameter $\widehat p$ for $p\le p_{\mathsf c}$, while it percolates for $p>p_{\mathsf c}$.

We also establish negative dependence across wired branches. On a finite wired tree, the vector of branch-connectivity indicators satisfies conditional negative association under positive external fields (CNA+). Hence bounded increasing observables supported on disjoint collections of branches incident to a common vertex have nonpositive covariance. The same inequality holds in the unique infinite-volume wired measure for every $p$, with equality for $p\le p_{\mathsf c}$.
\end{abstract}

\maketitle
\section{Introduction}

The random-cluster model, introduced by Fortuin and Kasteleyn, is a probability measure on random subgraphs that unifies independent percolation with the ferromagnetic Ising and Potts models~\cite{fortuinKasteleyn1972random,grimmett2006randomcluster}. On a finite graph $G=(V,E)$ with parameters $p\in(0,1)$ and $q>0$, it assigns to each configuration $\omega\in\{0,1\}^E$ the probability
\[
    \mu_{G,p,q}(\omega)
    =
    \frac{
        p^{|\omega|}(1-p)^{|E|-|\omega|}q^{\kappa(\omega)}
    }{
        Z_{G,p,q}
    },
\]
where $Z_{G,p,q}$ is the partition function, $|\omega|$ is the number of open edges, and $\kappa(\omega)$ is the number of connected components of the open subgraph. When $q=1$, this is ordinary independent bond percolation; when $q\ge2$ is an integer, it is coupled to the ferromagnetic $q$-state Potts model through the Fortuin--Kasteleyn--Edwards--Sokal representation~\cite{edwardsSokal1988}. The global dependence created by the cluster-counting term places the model at the intersection of probability, statistical physics, combinatorics, and algorithms. In particular, phase transitions and correlation decay on infinite trees have played an important role in the design and analysis of sampling algorithms on sparse random graphs~\cite{blancaEtAl2020sampling,blancaGheissari2021treeUniqueness,blanca2023sampling}.

The behavior of the model changes qualitatively as $q$ crosses one. For $q\ge1$, the FKG inequality gives positive association and a monotone ordering of boundary conditions~\cite{fortuinKasteleynGinibre1971,grimmett2006randomcluster}; these are among the principal tools in the finite- and infinite-volume theory. For $0<q<1$, by contrast, the dependence is expected to be negative. Negative association, in the sense introduced by Joag-Dev and Proschan~\cite{joagDevProschan1983}, and its conditional and external-field variants form a hierarchy of substantially stronger properties than pairwise negative correlation; see~\cite{borceaBrandenLiggett2009,kahnNeiman2010negative,pemantle2000towards} for foundational developments. For the random-cluster model with $0<q<1$, even pairwise negative correlation between two distinct edges remains open on general graphs~\cite{grimmettWinkler2004negative,pemantle2000towards}. Kahn and Neiman conjectured that every random-cluster measure with $0<q<1$ is NA+~\cite{kahnNeiman2010negative}. Our result does not establish this conjecture, but proves a branch-separated negative-association for the wired tree measure. Recent work on lattices improves the available stochastic comparisons for $q<1$ and proves uniqueness of infinite-volume measures in corresponding subcritical and supercritical parameter ranges~\cite{beffaraFaipeurOke2025}; these results concern lattice phase-transition bounds rather than the exact wired-tree DLR classification considered in this paper.

There are, however, important positive results on special graph classes and for related algebraic notions. Semple and Welsh, and independently the Potts--Rayleigh framework developed by Wagner, established pairwise negative-correlation inequalities under arbitrary positive edge weights for series--parallel graphs and more general classes obtained through matroid sums~\cite{sempleWelsh2008negative,wagner2008mason}. On a different front, Br{\"a}nd{\'e}n and Huh proved that the homogenized multivariate Tutte polynomial is Lorentzian when $0<q\le1$~\cite{brandenHuh2020lorentzian}, while strong log-concavity has led to an FPRAS for the partition function of the matroid random-cluster model when $0<q<1$, as well as efficient sampling of random spanning forests~\cite{anariEtAl2024logConcave}. These results reveal substantial algebraic structure below one, but they do not by themselves establish negative association for arbitrary disjoint edge sets.

Trees are a natural setting in which to investigate both the phase structure and the conjectural negative dependence. With free boundary conditions, the random-cluster measure on a finite tree is simply a product measure. Wiring the leaves together changes the geometry fundamentally: after contraction of the wired boundary, distinct branches become parallel routes to a common vertex, producing cycles and genuine dependence. Finite wired trees are series--parallel, so Wagner's Potts--Rayleigh theorem gives pairwise negative correlation between individual edges under arbitrary positive edge activities~\cite{wagner2008mason}; see also~\cite{sempleWelsh2008negative}. It does not, however, control arbitrary increasing observables supported on entire branch collections.

A closely related limiting model is the arboreal gas, obtained from the random-cluster model as $q\to0$ under the scaling $p=\beta q$. On regular trees with wired boundary conditions, Easo proved existence and exhaustion independence of the weak limit, while Ray and Xiao independently obtained a detailed description for the standard wired regular-tree exhaustion. In both works, the subcritical and critical limit is Bernoulli bond percolation, whereas the supercritical phase has additional infinite forest structure~\cite{easo2022wired,rayXiao2022forests}. These results provide a structural precursor to the same phase dichotomy at fixed $q\in(0,1)$. Our setting addresses uniqueness for the wired DLR specification, rather than only the wired finite-volume limit.

The wired-tree phase diagram goes back to H{\"a}ggstr{\"o}m's foundational work~\cite{haggstrom1996homogeneous}. Grimmett's analysis of regular trees constructs the all-wired limit for every $q>0$, computes the tree critical point and percolation probability, proves that the wired limit satisfies the wired DLR specification, and identifies the unique product wired state whenever the percolation probability vanishes~\cite{grimmett2006randomcluster}. The book works explicitly with the binary tree and notes that the corresponding statements hold for general $d$-ary trees. Its subsequent analysis of the remaining uniqueness question assumes $q\ge1$. For $q>2$, H{\"a}ggstr{\"o}m identified a non-uniqueness interval between two tree thresholds and conjectured that uniqueness returns above the upper threshold. Uniqueness sufficiently deep in the low-temperature regime was subsequently proved in~\cite{grimmettJanson2005branching,jonasson1999general}, and the conjectured transition was recently established throughout the full range $q>1$ in~\cite{blanca2026uniqueness}. Tree recursions and fixed-point analysis below the lower uniqueness threshold were also developed in connection with sampling on random regular and unbounded-degree random graphs~\cite{blancaEtAl2020sampling,blanca2023sampling}. The message parametrization and the sufficiently wired proof architecture are adapted from earlier joint work of the author~\cite{blanca2026uniqueness}. The genuinely new $q<1$ inputs are the amplification argument of Lemma~\ref{lem:density-amplification} and the uniform branch-survival argument of Lemma~\ref{lem:uniform-branch-survival}, which replace the FKG and stochastic-domination tools available for $q>1$. The remaining $q<1$ problem is therefore supercritical uniqueness: FKG comparison is unavailable, and one must rule out wired DLR measures other than the all-wired limit.

Our first contribution resolves this supercritical uniqueness problem and combines it with the subcritical--critical description to give the full wired DLR phase diagram for $0<q<1$. We also establish a corresponding negative-dependence theorem. Fix $\Delta\ge3$, put $d=\Delta-1$, and define
\[
    \widehat p
    :=
    \frac{p}{p+q(1-p)},
    \qquad
    \pc
    :=
    \frac{q}{d+q-1}
    =
    \frac{q}{\Delta+q-2}.
\]
Equivalently, $\pc$ is characterized by $d\widehat p=1$.

\begin{theorem*}[Supercritical wired-DLR uniqueness, informal]
Let $0<q<1$, $\Delta\ge3$, and $p>\pc$. The wired DLR specification on the infinite $\Delta$-regular tree admits a unique infinite-volume Gibbs measure, namely the local limit $\mu^*$ of the random-cluster measures on finite balls with all-wired boundary conditions.
\end{theorem*}

The existence of $\mu^*$ and the positivity of its percolation probability are known~\cite{grimmett2006randomcluster}. The new result is uniqueness: no additional wired DLR state appears above the transition when $0<q<1$. Combining this with the classical description at and below the critical point gives the full phase diagram.

\begin{corollary*}[Wired DLR phase diagram, informal]
Let $0<q<1$, $\Delta\ge3$, and $p\in(0,1)$. The wired DLR specification on the infinite $\Delta$-regular tree admits a unique infinite-volume Gibbs measure. Moreover:
\begin{enumerate}[(i)]
    \item if $p\le\pc$, this measure is the product measure
    \[
        \bigotimes_{e\in E(\mathcal T_\infty)}
        \operatorname{Ber}(\widehat p);
    \]
    \item if $p>\pc$, it is the local limit $\mu^*$ of the random-cluster measures on finite balls with all-wired boundary conditions, and
    \[
        \mu^*(\rho\leftrightarrow\infty)>0.
    \]
\end{enumerate}
\end{corollary*}

We include a direct proof of part~(i) in our notation, while Proposition~\ref{prop:wired-limit-percolates} records an explicit lower bound for the percolation probability in part~(ii). This uniqueness concerns the wired-at-infinity DLR specification. For every $p$, the free finite-volume measures on the tree converge to the product state $\bigotimes_e\operatorname{Ber}(\widehat p)$; when $p>p_{\mathsf c}$, this state percolates and is distinct from, and does not satisfy, the wired DLR state, so strong uniqueness across boundary conventions fails.

Our second main result identifies a robust form of negative dependence under this unique measure. Let $T$ be a finite tree whose leaves belong to one wired boundary class, and fix an internal vertex $x$. After contracting the wired boundary, the components incident to $x$ form edge-disjoint two-terminal branches $H_1,\ldots,H_m$ in parallel between $x$ and the wired vertex. We prove substantially more than pairwise negative correlation between their connection events.

\begin{theorem*}[Branch negative association, informal]
Let $0<q<1$. If $I,J\subseteq[m]$ are disjoint and $F$ and $G$ are bounded increasing observables supported, respectively, on the edge collections
\[
    \bigcup_{i\in I}E(H_i)
    \qquad\text{and}\qquad
    \bigcup_{j\in J}E(H_j),
\]
then
\[
    \operatorname{Cov}_{\mu_T^{\mathbf 1}}(F,G)
    \le0.
\]
The same inequality holds for bounded increasing cylinder observables supported on disjoint collections of branches incident to a common vertex under the unique infinite-volume wired Gibbs measure. At and below $\pc$, the covariance is zero.
\end{theorem*}

See Theorem~\ref{thm:branch-separated-na} and Corollary~\ref{cor:infinite-branch-na}. For a finite wired tree, the full vector of branch-connection indicators satisfies CNA+ under coordinate conditioning and positive external fields. The infinite-volume conclusion is the branch-separated covariance inequality, rather than CNA+ for an infinite branch vector. The separation condition is essential: the result does not cover arbitrary disjoint supports lying in the same branch.

For uniqueness, we use an exact message, or boundary law, attached to each rooted subtree, following the random-cluster tree recursions in~\cite{blancaEtAl2020sampling,blanca2023sampling}. Homogeneous messages evolve under a scalar map with the trivial fixed point $1$. At and below $\pc$, this is the only fixed point, and Bernoulli domination identifies the unique wired DLR measure with the product state. Above $\pc$, the fixed point $1$ is unstable and a unique attracting fixed point $x^*>1$ appears. Homogeneous attractivity alone is not enough for Gibbs uniqueness, since the boundary induced by an arbitrary exterior configuration is random and nonhomogeneous. We therefore first show that a positive density of wired leaves amplifies the messages away from $1$ and then prove contraction once the messages enter a fixed positive basin around $x^*$. This gives weak spatial mixing under positive-density wired boundary conditions. The remaining bridge is the uniform branch-survival estimate: every wired DLR measure induces such a positive-density boundary condition on distant balls, so its local marginals agree with those of the all-wired limit. The proof follows the sufficiently wired architecture for $q>1$ in~\cite{blanca2026uniqueness}, while the amplification and branch-survival arguments replace the FKG comparisons.

For negative association, we first analyze the indicators that the parallel branches connect their two terminals. Their joint law is a positive-external-field tilt of an exchangeable measure whose rank sequence is ultra-log-concave for $0<q<1$. Pemantle's exchangeable-measure criterion therefore gives CNA+~\cite{pemantle2000towards}; see also~\cite{kahnNeiman2010negative} for the role of exchangeability and log-concavity in negative dependence. Conditional on the connection vector, configurations in distinct branches are independent. A separate series--parallel argument shows that, within each branch, conditioning on terminal connection stochastically dominates conditioning on terminal disconnection. The law of total covariance combines this stochastic order with the indicator-level CNA+ property to prove the branch-separated inequality for arbitrary increasing observables. Local convergence and uniqueness then pass the covariance inequality to the infinite-volume wired measure. This argument does not rely on real stability, and the general strong-Rayleigh theory~\cite{borceaBrandenLiggett2009} does not directly cover these wired random-cluster laws.

The paper is organized as follows. Section~\ref{sec:model} introduces the finite-volume model, the wired DLR specification, and the negative-dependence terminology. Section~\ref{sec:message-fixed-points} derives the message recursion and classifies its fixed points. Section~\ref{sec:wsm} proves weak spatial mixing under positive-density wired boundary conditions. Section~\ref{sec:uniqueness} proves uniqueness, first above $\pc$ through boundary-law contraction and then at and below $\pc$ by identifying the Bernoulli product state. Section~\ref{sec:negative-dependence} proves CNA+ for finite branch-connectivity vectors, establishes finite branch-separated negative association, and passes the latter covariance inequality to the infinite-volume wired measure.

\section{Preliminaries}
\label{sec:model}

In this section we introduce the finite-volume random-cluster model, boundary conditions on finite rooted trees, and the wired infinite-volume specification used throughout the paper. We also record the notions of spatial mixing and negative dependence that will be used later.

\subsection{The finite-volume random-cluster model}

Let $G=(V,E)$ be a finite graph. A configuration is an element $\omega\in\{0,1\}^E$, where an edge $e$ is called open when $\omega(e)=1$ and closed when $\omega(e)=0$. We write
\[
    |\omega|
    :=
    \sum_{e\in E}\omega(e)
\]
for the number of open edges.

For parameters $p\in(0,1)$ and $q>0$, the random-cluster weight of $\omega$ is 
\[
    p^{|\omega|}(1-p)^{|E|-|\omega|}q^{\kappa(\omega)},
\] where $\kappa(\omega)$ is the number of connected components of the open subgraph $(V,\{e:\omega(e)=1\})$. The random-cluster measure on $G$ is
\[
    \mu_{G,p,q}(\omega)
    :=
    \frac{1}{Z_{G,p,q}}
    p^{|\omega|}
    (1-p)^{|E|-|\omega|}
    q^{\kappa(\omega)},
    \qquad
    \omega\in\{0,1\}^E,
\]
where
\[
    Z_{G,p,q}
    :=
    \sum_{\omega\in\{0,1\}^E}
    p^{|\omega|}
    (1-p)^{|E|-|\omega|}
    q^{\kappa(\omega)}
\]
is the partition function.

Fix a distinguished boundary set $\Lambda\subseteq V$. A boundary condition on $\Lambda$ is a partition $\xi$ of $\Lambda$. Given $\omega$, declare all vertices in each part of $\xi$ to be connected; equivalently, identify them as a single vertex. The free boundary condition, denoted by $\mathbf 0$, is the partition into singletons, so that $\mu_{G,p,q}^{\mathbf 0}=\mu_{G,p,q}$. The all-wired boundary condition, denoted by $\mathbf 1$, consists of one part containing every vertex of $\Lambda$. Let $\kappa^\xi(\omega)$ be the resulting number of open connected components.

The random-cluster measure on $G$ with boundary condition $\xi$ is
\[
    \mu_{G,p,q}^\xi(\omega)
    :=
    \frac{1}{Z_{G,p,q}^\xi}
    p^{|\omega|}
    (1-p)^{|E|-|\omega|}
    q^{\kappa^\xi(\omega)},
    \qquad
    \omega\in\{0,1\}^E,
\]
where
\[
    Z_{G,p,q}^\xi
    :=
    \sum_{\omega\in\{0,1\}^E}
    p^{|\omega|}
    (1-p)^{|E|-|\omega|}
    q^{\kappa^\xi(\omega)}
\]
is the partition function.

Throughout the paper, when the underlying graph and the parameters $p,q$ are clear, they will be omitted from the notation.

\begin{remark}
For $q\ge1$, the random-cluster model satisfies the FKG inequality and is monotone in the boundary condition. Neither property is available in the regime $0<q<1$. The supercritical uniqueness proof therefore obtains the required boundary comparisons by different arguments rather than by monotonicity.
\end{remark}

\subsection{Trees and basic notation}

Fix an integer $\Delta\ge3$, and write $d:=\Delta-1$. For a finite rooted tree $T^\rho$, the vertex $\rho$ denotes its root, and the boundary $\partial T^\rho$ is the set of leaves of $T^\rho$. For $u\in V(T^\rho)$, let $T^u$ denote the descendant subtree rooted at $u$. Unless otherwise specified, for a finite rooted tree $T$, $\partial T$ denotes its set of leaves, and boundary conditions are imposed on $\partial T$.

Let $\mathcal T_\infty$ denote the infinite $\Delta$-regular tree rooted at $\rho$. Every vertex other than $\rho$ has one parent and $d$ children, while $\rho$ has $\Delta$ children. For $h\ge1$, let $\dary$ denote the complete $d$-ary tree of height $h$ rooted at $\rho$, and let $\dreg$ denote the finite rooted $\Delta$-regular tree of height $h$ rooted at $\rho$. Equivalently, $\dreg$ is the ball of radius $h$ around $\rho$ in $\mathcal T_\infty$.

For vertices or edge sets, $\dist(\,\cdot\,,\,\cdot\,)$ denotes graph distance. If $\Lambda$ is a finite edge set and $A$ is a vertex set, we use the convention
\[
    \dist(\Lambda,A)
    :=
    \min\bigl\{
        \dist(v,A):
        v\text{ is an endpoint of an edge in }\Lambda
    \bigr\}.
\]
For an edge configuration $\omega$, we write $x\stackrel{\omega}{\longleftrightarrow}y$ when $x$ and $y$ are connected by an open path. The superscript $\omega$ will be omitted when the configuration is clear.

\subsection{The critical point}
\label{subsec:critical-point}

The critical point $\pc$ used throughout the paper is determined by the effective one-edge parameter
\[
    \widehat p
    :=
    \frac{p}{p+q(1-p)}.
\]
To see the one-edge rule directly, condition on all edges other than $e$. If the endpoints of $e$ are already connected through the conditioned configuration and the boundary identification, opening $e$ does not change the number of clusters, so the conditional open probability is $p$. Otherwise, opening $e$ merges two clusters, and the relative weights of the open and closed states are $p$ and $q(1-p)$, respectively. The conditional open probability is therefore $\widehat p$. Since $0<q<1$, one has $\widehat p>p$. Thus $\widehat p$, rather than $p$, is the natural open-edge parameter when exploring a cluster through portions of the tree that are not already connected through the wired exterior.

Every non-root vertex of $\mathcal T_\infty$ has $d=\Delta-1$ forward edges. The corresponding branching parameter is therefore $d\widehat p$, which motivates defining the tree critical point by $d\widehat p=1$. Solving this equation for $p$ gives
\begin{equation}
\label{eq:critical-point}
    \pc
    :=
    \frac{q}{d+q-1}.
\end{equation}
Equivalently,
\[
    p\le\pc
    \quad\Longleftrightarrow\quad
    d\widehat p\le1,
    \qquad
    p>\pc
    \quad\Longleftrightarrow\quad
    d\widehat p>1.
\]
This branching interpretation agrees exactly with the message analysis below. In Section~\ref{subsec:fixed-points}, the homogeneous recursion has the trivial fixed point $x=1$, corresponding to zero connection probability to the wired boundary, and its derivative there is $d\widehat p$.

\subsection{Local limits and the wired DLR specification}
Let $\Omega_\infty:=\{0,1\}^{E(\mathcal T_\infty)}$ with the product $\sigma$-algebra. An event or function is called \emph{local}, or a \emph{cylinder event} or \emph{cylinder function}, if it depends on only finitely many edges.

A sequence of probability measures $(\nu_n)_{n\ge1}$ on $\Omega_\infty$ converges locally to a probability measure $\nu$ if, for every finite edge set $\Lambda$, $\nu_n\vert_\Lambda\longrightarrow\nu\vert_\Lambda$ in total variation. Equivalently, $\mathbb E_{\nu_n}[F]\longrightarrow\mathbb E_\nu[F]$ for every bounded cylinder function $F$.

We next define the infinite-volume wired specification independently of any limiting construction.
Let $T\subset\mathcal T_\infty$ be a finite connected subtree and write $E_T:=E(T)$ and $\mathcal F_{T^c}:=\sigma\bigl(\omega(e):e\in E(\mathcal T_\infty)\setminus E_T\bigr)$. For $\omega\in\Omega_\infty$, let $\omega_{T^c}$ denote its restriction to the edges outside $E_T$. The exterior vertex boundary of this finite region is $\partial_{\mathrm E}T:=\bigl\{v\in V(T):v\text{ is incident to an edge in }E(\mathcal T_\infty)\setminus E_T\bigr\}$.

Fix an exterior configuration $\eta\in\{0,1\}^{E(\mathcal T_\infty)\setminus E_T}$. The exterior configuration induces a boundary partition $\xi_T(\eta)$ on $\partial_{\mathrm E}T$ as follows. Consider the equivalence relation generated by the following two rules:
\begin{enumerate}[(i)]
    \item $x,y\in\partial_{\mathrm E}T$ are equivalent if they are connected by an $\eta$-open path using only edges outside $E_T$;
    \item $x,y\in\partial_{\mathrm E}T$ are equivalent if each is connected through the exterior configuration to an infinite open component.
\end{enumerate}
Thus all vertices of $\partial_{\mathrm E}T$ connected through the exterior configuration to infinity belong to one common wired class, even when they lie in distinct infinite open components.

Let $\gamma_T(\cdot\mid\eta):=\mu_T^{\xi_T(\eta)}$, viewed as a kernel on $\Omega_\infty$ by leaving the exterior configuration unchanged.

\begin{lemma}[Consistency of the wired kernels]
\label{lem:wired-kernel-consistency}
The family $(\gamma_T)_T$ is a measurable, proper, and consistent probability specification. In particular, if $T_1\subseteq T_2$ are finite connected subtrees, then
\[
    \gamma_{T_2}\gamma_{T_1}
    =
    \gamma_{T_2}.
\]
\end{lemma}

\begin{proof}
For $x\in\partial_{\mathrm E}T$, connection to infinity through the exterior equals the intersection, over $n\ge1$, of the cylinder events that $x$ is connected through exterior edges to distance $n$. Hence $\eta\mapsto\xi_T(\eta)$ is measurable, and so is $\gamma_T$. Properness follows from the convention that $\gamma_T$ leaves the exterior configuration unchanged.

For consistency, fix $T_1\subseteq T_2$, an exterior configuration $\eta$ outside $E_{T_2}$, and a configuration $\zeta$ on $E_{T_2}\setminus E_{T_1}$. The partition induced on $\partial_{\mathrm E}T_1$ by the combined exterior configuration $(\zeta,\eta)$ is exactly the partition obtained by starting from $\xi_{T_2}(\eta)$ and then revealing $\zeta$. Indeed, two boundary vertices of $T_1$ are joined precisely when they are connected through the revealed annulus and the partition on $\partial_{\mathrm E}T_2$; moreover, a boundary vertex of $T_1$ reaches infinity precisely when it reaches the wired class of $\xi_{T_2}(\eta)$. The finite-volume domain Markov property therefore identifies the conditional law in $T_1$ under $\mu_{T_2}^{\xi_{T_2}(\eta)}$ with $\mu_{T_1}^{\xi_{T_1}(\zeta,\eta)}$. Averaging over $\zeta$ gives the stated consistency identity.
\end{proof}

This is the wired specification described in~\cite[Equation~(10.66)]{grimmett2006randomcluster}.

\begin{definition}[Wired DLR random-cluster measure]
\label{def:dlr}
A probability measure $\nu$ on $\Omega_\infty$ is called an \emph{infinite-volume random-cluster measure satisfying the wired DLR condition} if, for every finite connected subtree $T\subset\mathcal T_\infty$ and every bounded function $F$ depending only on the edges of $T$,
\[
    \mathbb E_\nu
    \left[
        F
        \,\middle|\,
        \mathcal F_{T^c}
    \right]
    =
    \gamma_T
    (F\mid\omega_{T^c})
    \qquad
    \nu\text{-a.s.}
\]
\end{definition}

In other words, conditionally on the configuration outside $T$, the law inside $T$ is the finite-volume random-cluster measure with the boundary partition induced by exterior connections, with all infinite exterior components wired together at infinity.

Local convergence and the DLR property are separate issues. Even when finite-volume measures converge locally, the limit does not automatically satisfy the DLR condition because the induced boundary partition depends on connections to infinity through the exterior configuration; this is related to the failure of almost-sure quasilocality for random-cluster measures on trees~\cite{haggstrom1996quasilocality}. In Section~\ref{sec:uniqueness}, we prove local convergence of the all-wired measures directly and record their DLR property separately before proving uniqueness.

\subsection{Negative-dependence terminology}

Following Pemantle~\cite{pemantle2000towards}, we introduce the negative-dependence notions used later. Let $\mu$ be a probability measure on $\{0,1\}^E$, with $E$ finite. We write $\omega\le\omega'$ if $\omega(e)\le\omega'(e)$ for every $e\in E$, that is, if every edge open in $\omega$ is also open in $\omega'$. A function $F:\{0,1\}^E\to\mathbb R$ is increasing if $\omega\le\omega' \Longrightarrow F(\omega)\le F(\omega')$.

\begin{definition}
The measure $\mu$ is \emph{negatively associated}, abbreviated \emph{NA}, if for every pair of disjoint edge sets $E_1,E_2\subseteq E$ and every pair of bounded increasing functions $F=F(\omega_{E_1})$ and $G=G(\omega_{E_2})$, one has
\[
    \operatorname{Cov}_\mu(F,G)
    \le0.
\]
\end{definition}

In particular, negative association implies pairwise negative correlation $\operatorname{Cov}_\mu \bigl(\omega(e),\omega(f)\bigr) \le0$ for edges $e\ne f$. The converse is false in general.

For positive numbers $\mathbf a=(a_e)_{e\in E}$, the external-field tilt of $\mu$ is the measure
\[
    \mu^{\mathbf a}(\omega)
    :=
    \frac{1}{Z(\mathbf a)}
    \mu(\omega)
    \prod_{e\in E}a_e^{\omega(e)}.
\]
For a random-cluster measure, this is equivalent to replacing the homogeneous edge activity $p/(1-p)$ by inhomogeneous positive edge activities.

A binary measure has the \emph{CNA+ property} if every measure obtained from it by a positive external-field tilt followed by a coordinate projection is conditionally negatively associated; that is, it remains negatively associated after every further positive-probability coordinate pinning. Here CNA stands for conditional negative association, while the $+$ indicates stability under positive external fields and projections. We will use the hierarchy in~\cite{pemantle2000towards}
\[
    \mathrm{CNA+}
    \quad\Longrightarrow\quad
    \mathrm{NA}
    \quad\Longrightarrow\quad
    \text{pairwise negative correlation}.
\]

We will not require full negative association of the random-cluster measure on arbitrary edge sets. Instead, the tree recursion produces CNA+ laws for certain branch-connectivity indicators. These local laws will imply negative association for increasing observables supported in disjoint collections of branches separated by a common vertex.

\section{Message recursion and fixed points}
\label{sec:message-fixed-points}
\subsection{Boundary conditions and messages}
\label{subsec:boundary-messages}

Throughout this section, we fix a finite rooted tree $T^\rho$ with root $\rho$. For a vertex $u\in V(T^\rho)$, we write $T^u$ for the rooted subtree below $u$. Thus $T^\rho$ denotes the whole tree. We also write $\partial T^\rho$ for the set of leaves. $\mathbf 1$ denotes the all-wired boundary condition on $\partial T^\rho$.

A boundary condition on $\partial T^\rho$ is a partition of the boundary vertices; vertices in the same part are considered wired together. The boundary conditions used in the tree recursion are the following.

\begin{definition}
\label{def:single-component-bc}
A boundary condition $\xi$ on $\partial T^\rho$ is called \emph{single-component} if all parts of its boundary partition except possibly one are singletons. We denote its specified wired component by $\mathcal C_1(\xi)$ and regard it as part of the boundary-condition data. The set $\mathcal C_1(\xi)$ may be empty or a singleton; when it has at least two vertices, it is the unique non-singleton part of the partition.
\end{definition}

The all-wired boundary condition is the main example of a single-component boundary condition. More generally, this notion allows us to track the wired boundary component through the tree.

For $u\in V(T^\rho)$, let $\xi_u$ be the restriction of the boundary partition $\xi$ to $\partial T^u$, with $\mathcal C_1(\xi_u):=\mathcal C_1(\xi)\cap\partial T^u$. Thus $\xi_u$ is again a single-component boundary condition, and its wired component may be empty.

We split the partition function on $T^u$ according to whether $u$ is connected to $\mathcal C_1(\xi_u)$. Let $Z_{T^u}^{\xi_u,1}$ be the restricted partition function for configurations in which $u$ is connected to $\mathcal C_1(\xi_u)$, and let $Z_{T^u}^{\xi_u,0}$ be the restricted partition function for configurations in which $u$ is not connected to $\mathcal C_1(\xi_u)$. If $\mathcal C_1(\xi_u)=\varnothing$, we set $Z_{T^u}^{\xi_u,1}:=0$ and $Z_{T^u}^{\xi_u,0}:=Z_{T^u}^{\xi_u}$.

\begin{definition}[Message]
\label{def:message}
For $u\in V(T^\rho)$, define the message at $u$ by
\[
    f^{\xi}(u)
    :=
    1+
    q
    \frac{
        Z_{T^u}^{\xi_u,1}
    }{
        Z_{T^u}^{\xi_u,0}
    }.
\]
For a boundary vertex $u\in \partial T^\rho$, we use the convention
\[
    f^{\xi}(u)
    =
    \begin{cases}
    1, & u\notin \mathcal C_1(\xi),\\
    +\infty, & u\in \mathcal C_1(\xi).
    \end{cases}
\]
\end{definition}

Thus internal messages take values in $[1,\infty)$, while boundary messages may take the value $+\infty$. The value $1$ means that there is no connection bias toward the wired component, while values larger than $1$ correspond to a positive connection bias.

Because the message is a normalized ratio, it immediately recovers the corresponding connection probability.  The conversion is uniformly Lipschitz on $[1,\infty)$, so later message-contraction estimates transfer directly to connection probabilities.

\begin{fact}
\label{fact:message-probability}
For every internal vertex $u\in V(T^\rho)$,
\[
    \mu_{T^u}^{\xi_u}
    \bigl(u\sim \mathcal C_1(\xi_u)\bigr)
    =
    \frac{
        f^{\xi}(u)-1
    }{
        f^{\xi}(u)+q-1
    }.
\]
Consequently, if $\xi_1,\xi_2$ are two single-component boundary conditions, then
\[
\begin{aligned}
&
\left|
    \mu_{T^u}^{(\xi_1)_u}
    \bigl(u\sim \mathcal C_1((\xi_1)_u)\bigr)
    -
    \mu_{T^u}^{(\xi_2)_u}
    \bigl(u\sim \mathcal C_1((\xi_2)_u)\bigr)
\right|
\le
\frac{1}{q}
\left|
    f^{\xi_1}(u)
    -
    f^{\xi_2}(u)
\right|.
\end{aligned}
\]
\end{fact}

\begin{proof}
When $\mathcal C_1(\xi_u)=\varnothing$, the connection event is understood as the empty event, and both sides of the first identity are zero. Otherwise, Definition~\ref{def:message} gives
\[
    \frac{
        Z_{T^u}^{\xi_u,1}
    }{
        Z_{T^u}^{\xi_u,0}
    }
    =
    \frac{
        f^{\xi}(u)-1
    }{q}.
\]
Therefore
\[
\begin{aligned}
    \mu_{T^u}^{\xi_u}
    \bigl(u\sim \mathcal C_1(\xi_u)\bigr)
    =
    \frac{
        Z_{T^u}^{\xi_u,1}
    }{
        Z_{T^u}^{\xi_u,0}
        +
        Z_{T^u}^{\xi_u,1}
    }
    =
    \frac{
        f^{\xi}(u)-1
    }{
        f^{\xi}(u)+q-1
    }.
\end{aligned}
\]

The second claim follows by subtracting the two displayed expressions. Indeed,
\[
\left|
    \frac{f^{\xi_1}(u)-1}
         {f^{\xi_1}(u)+q-1}
    -
    \frac{f^{\xi_2}(u)-1}
         {f^{\xi_2}(u)+q-1}
\right| =
\left|
    \frac{
        q\bigl(f^{\xi_1}(u)
        -
        f^{\xi_2}(u)\bigr)
    }
    {
        \bigl(f^{\xi_1}(u)+q-1\bigr)
        \bigl(f^{\xi_2}(u)+q-1\bigr)
    }
\right|.
\]
Since each message is at least $1$, each denominator factor is at least $q$.  Hence the right-hand side is at most $q|f^{\xi_1}(u)-f^{\xi_2}(u)|/q^2$, which proves the stated bound. 
\end{proof}

To control the effect of changing the boundary condition, it is enough to control the corresponding messages. This is the reason the message recursion is the main object of the analysis.
We now derive the recursion satisfied by the messages. Define
\begin{equation}
\label{eq:Phi}
\Phi(x)
:=
\frac{
    x+(q-1)(1-p)
}{
    (1-p)x+p+(q-1)(1-p)
},
\qquad x\in[1,\infty),
\end{equation}
with the convention $\Phi(\infty)=(1-p)^{-1}$.
For $x\ge1$, the denominator is at least $p+q(1-p)>0$, and $\Phi(x)\ge1$.

\begin{lemma}[Tree message recursion]
\label{lem:tree-message-recursion}
Let \(T^\rho\) be a finite rooted tree and let \(\xi\) be a single-component boundary condition on \(\partial T^\rho\). For every internal vertex \(u\),
\[
    f^{\xi}(u)
    =
    \prod_{w\in N(u)}
    \Phi\bigl(f^{\xi}(w)\bigr),
\]
where \(N(u)\) is the set of children of \(u\).
\end{lemma}

\begin{proof}
The same recursion appears as Lemma~2.5 of~\cite{blanca2026uniqueness}. For completeness, and because we use it for $0<q<1$, we give the direct calculation here. 

For a vertex $w$, write $Z_i(w):=Z_{T^w}^{\xi_w,i}$ for $i\in\{0,1\}$ and set $t:=p/q+(1-p)$. Define
\[
    A:=\prod_{w\in N(u)}
    \frac{tZ_0(w)+(1-p)Z_1(w)}{q},
    \qquad
    B:=\prod_{w\in N(u)}
    \frac{tZ_0(w)+Z_1(w)}{q}.
\]
Summing over the states of the edges from $u$ to its children gives $Z_{T^u}^{\xi_u,0}=q^2A$ and $Z_{T^u}^{\xi_u,1}=q(B-A)$. Therefore
\[
    f^\xi(u)=\frac BA=\prod_{w\in N(u)}
    \frac{tZ_0(w)+Z_1(w)}
         {tZ_0(w)+(1-p)Z_1(w)}
    =\prod_{w\in N(u)}
    \frac{f^\xi(w)+(q-1)(1-p)}
         {(1-p)f^\xi(w)+p+(q-1)(1-p)}.
\]
\end{proof}

\subsection{The homogeneous recursion and its fixed points}
\label{subsec:fixed-points}

We now specialize the recursion to the complete \(d\)-ary tree. When all children of a vertex receive the same message \(x\), by Lemma~\ref{lem:tree-message-recursion} the parent receives the message $\phi(x):=\Phi(x)^d$. Thus
\begin{equation}
\label{eq:phi}
    \phi(x)
    =
    \left(
    \frac{
        x+(q-1)(1-p)
    }{
        (1-p)x+p+(q-1)(1-p)
    }
    \right)^d,
    \qquad x\in[1,\infty),
\end{equation}
with the convention $\phi(\infty)=(1-p)^{-d}$.

Recall from~\eqref{eq:critical-point} that the tree threshold is
$\pc=q/(d+q-1)$.
The value \(x=1\) is always a fixed point of \(\phi\), since \(\Phi(1)=1\). This fixed point corresponds to zero connection probability to the wired boundary component.

The derivative at \(x=1\) is
\[
    \phi'(1)
    =
    \frac{dp}{p+q(1-p)}
    =
    d\widehat p.
\]
Therefore
\[
    \phi'(1)>1
    \quad\Longleftrightarrow\quad
    p>\pc.
\]
Thus, in the regime \(p>\pc\), the trivial fixed point \(1\) is unstable.

We now classify the fixed points of \(\phi\) on \([1,\infty)\) in the regime \(0<q<1\). The analysis is simpler than in the case \(q>1\): the lower and upper tree thresholds collapse, and there is no intermediate regime. For $q>1$, the scalar recursion can have multiple nontrivial fixed points, so the loss of stability of $x=1$ and the transition to the wired fixed point occur at distinct parameter values; see~\cite{blanca2026uniqueness} for a detailed fixed-point analysis. The two fixed-point regimes for $q<1$ are illustrated in Figure~\ref{fig:fixed-point-structure}.

Before starting the fixed-point analysis, we introduce the auxiliary function
\[
    \psi(r):=\frac{r^d+q-1}{\sum_{i=1}^{d}r^i+q-1},
    \qquad r>1.
\]
The next lemma establishes the basic properties of \(\psi\) that will be used in the following theorem.

\begin{lemma}
\label{lem:psi-monotone}
Let \(0<q<1\) and \(d\ge2\). Then \(\psi\) is strictly increasing on \((1,\infty)\). Moreover,
\[
    \lim_{r\downarrow1}\psi(r)
    =
    \frac{q}{d+q-1}
    =
    \pc,
    \qquad
    \lim_{r\to\infty}\psi(r)=1.
\]
\end{lemma}

\begin{proof}
Let $A(r):=\sum_{i=1}^{d}r^i$. Then, 
\[
    \psi(r)=\frac{r^d+q-1}{A(r)+q-1}.
\]
For \(r>1\), the denominator is positive. By differentiating, we obtain
\[
    \psi'(r)=\frac{dr^{d-1}(A(r)+q-1) - A'(r)(r^d+q-1)}{(A(r)+q-1)^2}.
\]
Since
\[
    A'(r)=\sum_{i=1}^{d} i r^{i-1},
\]
the numerator of $\psi'(r)$ can be written as
\[
    \sum_{i=1}^{d-1}
    r^{i-1}
    \Bigl((d-i)r^d+i(1-q)\Bigr).
\]
Every term in this sum is strictly positive because \(r>1\), \(d-i>0\), and \(1-q>0\). Hence $\psi'(r)>0$ for all $r>1$. The endpoint limits are immediate since the numerator and denominator of $\psi(r)$ both have leading term $r^d$, so their ratio tends to $1$ as $r\to\infty$.
\end{proof}

\begin{figure}[t]
\centering
\begin{subfigure}[t]{0.48\linewidth}
    \centering
    \begin{overpic}[
        width=\linewidth,
        clip,
        trim={0cm 0cm 2cm 0cm}
    ]{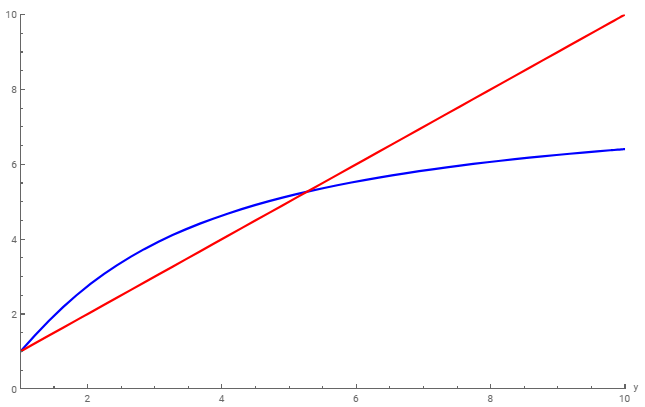}
        \put(76,58){%
            \makebox(0,0)[l]{\scriptsize$y=x$}%
        }
        \put(67,37){%
            \makebox(0,0)[l]{\scriptsize$y=\phi(x)$}%
        }
        \put(54.5,38.1){%
            \color{black}\circle*{1.2}%
        }
        \put(44,42){%
            \makebox(0,0)[l]{\scriptsize$(x^*,x^*)$}%
        }
        \put(27,70){%
            \makebox(0,0)[l]{\scriptsize$(p,q,d)=(0.5,0.5,3),\quad \pc=0.2$}%
        }
    \end{overpic}
    \caption{$p>\pc$}
\end{subfigure}
\hfill
\begin{subfigure}[t]{0.48\linewidth}
    \centering
    \begin{overpic}[
        width=\linewidth,
        clip,
        trim={0cm 0cm 2cm 0cm}
    ]{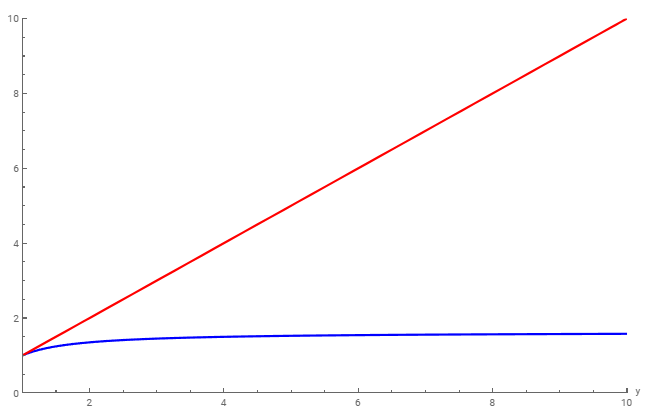}
        \put(76,58){%
            \makebox(0,0)[l]{\scriptsize$y=x$}%
        }
        \put(67,19){%
            \makebox(0,0)[l]{\scriptsize$y=\phi(x)$}%
        }
        \put(27,70){%
            \makebox(0,0)[l]{\scriptsize$(p,q,d)=(0.15,0.5,3),\quad \pc=0.2$}%
        }
    \end{overpic}
    \caption{$p<\pc$}
\end{subfigure}

\caption{Plots of $\phi(x)$ in the two regimes. In each panel, intersections of the blue curve with the red line are fixed points of $\phi$.}
\label{fig:fixed-point-structure}
\end{figure}
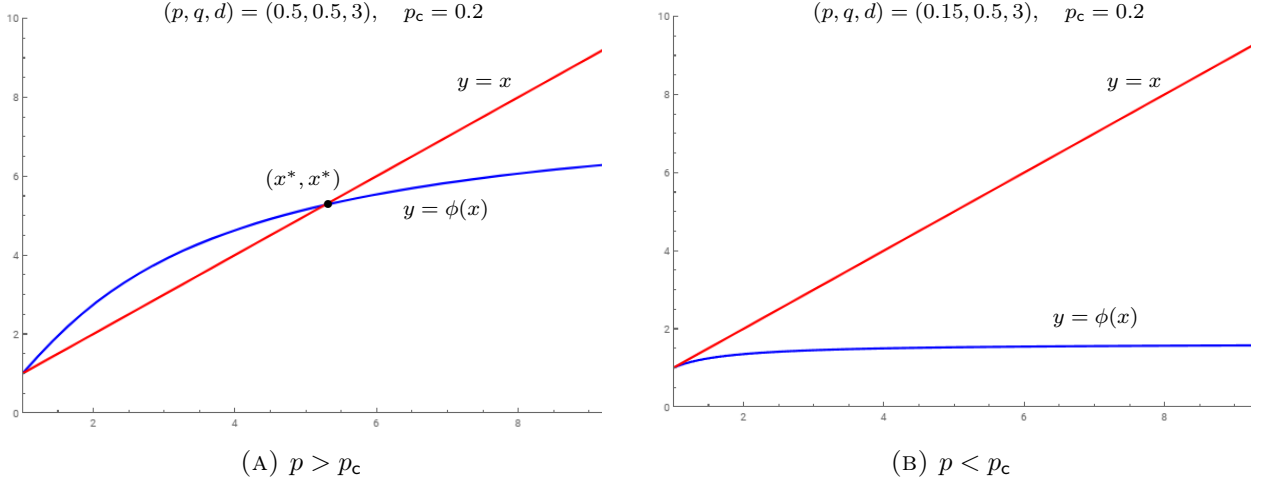

\begin{theorem}[Fixed-point structure for \(0<q<1\)]
\label{thm:fixed-point-structure}
Let \(0<q<1\), \(d\ge2\), and \(p\in(0,1)\). Let \(\phi\) be the function in~\eqref{eq:phi}. Then \(x=1\) is always a fixed point of \(\phi\). Moreover:
\begin{enumerate}[(i)]
    \item If \(p\le \pc\), then \(\phi\) has no fixed point in \((1,\infty)\). Hence \(x=1\) is the unique fixed point in \([1,\infty)\).

    \item If \(p>\pc\), then \(\phi\) has exactly one fixed point in \((1,\infty)\). Hence \(\phi\) has exactly two fixed points in \([1,\infty)\): the trivial fixed point \(1\), and a unique nontrivial fixed point \(x^*>1\).
\end{enumerate}
In the second case, $x^*=(r^*)^d$, where \(r^*>1\) is the unique solution of 
\[
    p=\psi(r)
    =
    \frac{
        r^d+q-1
    }{
        \sum_{i=1}^{d}r^i+q-1
    }.
\]
\end{theorem}

\begin{proof}
The identity \(\phi(1)=1\) follows from \(\Phi(1)=1\). Now let \(x>1\) and write \(r=x^{1/d}>1\). Then \(\phi(x)=x\) is equivalent to $r=\Phi(r^d)$. 
Equivalently, $r\bigl((1-p)r^d+p+(q-1)(1-p)\bigr)=r^d+(q-1)(1-p)$. After rearranging, we have $(r-1)(r^d+q-1)=p(r^{d+1}-r+(r-1)(q-1))$. Factoring out \(r-1\), this equation becomes $p=\psi(r)$, where
\begin{equation}
\label{eq:psi}
    \psi(r)
    :=
    \frac{
        r^d+q-1
    }{
        \sum_{i=1}^{d}r^i+q-1
    },
    \qquad r>1.
\end{equation}

By Lemma~\ref{lem:psi-monotone}, the function \(\psi\) is strictly increasing on \((1,\infty)\), with image $(\pc,1)$. Therefore, if \(p\le \pc\), there is no solution \(r>1\), and hence no fixed point \(x>1\). Thus \(x=1\) is the unique fixed point in \([1,\infty)\).

If \(p>\pc\), then since \(p<1\), there exists a unique \(r^*>1\) such that $p=\psi(r^*)$. The corresponding value $x^*=(r^*)^d$ is the unique fixed point of \(\phi\) in \((1,\infty)\).
\end{proof}

\subsection{Attractivity of the nontrivial fixed point}
\label{subsec:attractivity}

We now assume $p>\pc$. By Theorem~\ref{thm:fixed-point-structure}, the map \(\phi\) has exactly two fixed points in \([1,\infty)\): the trivial fixed point \(1\), and a unique nontrivial fixed point \(x^*>1\). The point \(1\) is unstable. We now prove that $x^*$ attracts every initial condition in $(1,\infty]$.

\begin{proposition}
\label{prop:attractivity}
Let \(0<q<1\), \(d\ge2\), and \(p>\pc\). Let \(x^*>1\) be the unique nontrivial fixed point of \(\phi\). Then for every \(x_0>1\),
\[
    \lim_{k\to\infty}\phi^{(k)}(x_0)=x^*.
\]
Here $\phi^{(k)}$ denotes the $k$-fold iterate of $\phi$.
The same conclusion holds for \(x_0=\infty\), with the convention \(\phi(\infty)=(1-p)^{-d}\). Moreover, $0<\phi'(x^*)<1$.
\end{proposition}

\begin{proof}
First observe that \(\Phi\), and hence \(\phi\), is strictly increasing on \([1,\infty)\). Also,
\[
    \lim_{x\to\infty}\phi(x)=(1-p)^{-d}<\infty.
\]

We claim that
\begin{equation}\label{eq:phi-crossing-sign}
    \phi(x)>x \quad\text{for }1<x<x^*,
    \qquad
    \phi(x)<x \quad\text{for }x>x^*.
\end{equation}
Indeed, write \(r=x^{1/d}\) and $\phi_p$ for the map $\phi$ to emphasize the dependence on $p$. For fixed \(x>1\),
\[
    \partial_p\Phi_p(x)
    =
    \frac{(x-1)(x+q-1)}
    {\bigl((1-p)x+p+(q-1)(1-p)\bigr)^2}
    >0,
\]
and hence $\phi_p(x)=\Phi_p(x)^d$ is strictly increasing in $p$. The equation \(\phi_p(x)=x\) is equivalent to \(p=\psi(r)\) where $\psi$ is the function in \eqref{eq:psi}. Since \(\psi\) is strictly increasing and \(p=\psi(r^*)\), where \(x^*=(r^*)^d\), we have
\begin{equation}\label{eq:p-psi-crossing-sign}
    p>\psi(r) \quad\text{if }1<r<r^*,
    \qquad
    p<\psi(r) \quad\text{if }r>r^*.
\end{equation}
Rearranging the fixed-point equation back in terms of \(\phi_p\), the inequalities in \eqref{eq:p-psi-crossing-sign} give \(\phi_p(x)>x\) for \(1<x<x^*\) and \(\phi_p(x)<x\) for \(x>x^*\), which is exactly \eqref{eq:phi-crossing-sign}.

Set $x_k:=\phi^{(k)}(x_0)$, so that $x_{k+1}=\phi(x_k)$. If \(x_0\in(1,x^*)\), then $x_1=\phi(x_0)>x_0$ by~\eqref{eq:phi-crossing-sign}. Since $\phi$ is increasing and $x_0<x^*$, we have $x_1=\phi(x_0)<\phi(x^*)=x^*$. Thus, $x_0<x_1<x^*$. By induction, $x_k<x_{k+1}<x^*$ for all $k\ge0$. Thus \((x_k)_{k\ge0}\) is increasing and bounded above by \(x^*\), so it converges to some $\tilde x \in (1,x^*]$. By continuity of \(\phi\),
\[
    \phi(\tilde x)
    =\phi\left(\lim_{k\to\infty}x_k\right)
    =\lim_{k\to\infty}\phi\left(x_k\right)
    =\lim_{k\to\infty}x_{k+1}
    =\tilde x.
\]
Since $\tilde x>1$ and $x^*$ is the unique fixed point in $(1, \infty)$, we have $\tilde x = x^*$.

If \(x_0>x^*\), then \(x_1=\phi(x_0)<x_0\). Since \(\phi\) is increasing and \(x_0>x^*\), we also have \(x_1=\phi(x_0)>\phi(x^*)=x^*\). Thus \(x^*<x_1<x_0\), and by induction \(x^*<x_{k+1}<x_k\) for all \(k\ge0\). The sequence $(x_k)_{k\ge0}$ is decreasing and bounded below by \(x^*\), so the same argument applies. 

If \(x_0=\infty\), then $x_1=\phi(\infty)=(1-p)^{-d}\in(1,\infty).$ This reduces to the case above.

It remains to prove \(0<\phi'(x^*)<1\). Positivity follows from the explicit derivative formula
\[
\phi'(x)
=
d p
\left(
\frac{
    x+(q-1)(1-p)
}{
    (1-p)x+p+(q-1)(1-p)
}
\right)^{d-1}
\frac{
    1+(q-1)(1-p)
}{
    \bigl((1-p)x+p+(q-1)(1-p)\bigr)^2
} > 0
\] 
for $x\ge1$. Thus \(\phi'(x^*)>0\).

For the upper bound, set $h(x):=\phi(x)-x$, and write $\Phi_p$ for $\Phi$ to emphasize its dependence on $p$. Define
\[
    D_p(r)
    :=
    (1-p)r^d+p+(q-1)(1-p)>0,
\]
for $r\ge1$. Using
\[
    r^{d+1}-r
    =
    (r-1)\sum_{i=1}^d r^i,
\]
we obtain
\[
\begin{aligned}
    D_p(r)\bigl(\Phi_p(r^d)-r\bigr)
    &=
    r^d+(q-1)(1-p)-rD_p(r)
    \\
    &=
    (r-1)
    \left[
        p\left(\sum_{i=1}^d r^i+q-1\right)
        -(r^d+q-1)
    \right]
    \\
    &=
    (r-1)
    \left(\sum_{i=1}^d r^i+q-1\right)
    \bigl(p-\psi(r)\bigr),
\end{aligned}
\]
where the last equality uses the definition of $\psi$ in~\eqref{eq:psi}. Dividing by $D_p(r)$ gives
\[
    \Phi_p(r^d)-r
    =
    \frac{
        (r-1)\left(\sum_{i=1}^d r^i+q-1\right)
    }{
        D_p(r)
    }
    \bigl(p-\psi(r)\bigr).
\]
Since $p=\psi(r^*)$, differentiation at $r=r^*$ yields
\[
\begin{aligned}
    \left.
    \frac{d}{dr}
    \bigl(\Phi_p(r^d)-r\bigr)
    \right|_{r=r^*}
    &=-
    \frac{
        (r^*-1)\left(\sum_{i=1}^d (r^*)^i+q-1\right)
    }{
        D_p(r^*)
    }
    \psi'(r^*)
    <0,
\end{aligned}
\]
where the last inequality follows from Lemma~\ref{lem:psi-monotone}. On the other hand, since $x^*=(r^*)^d$ and $\Phi_p(x^*)=r^*$,
\[
    \left.
    \frac{d}{dr}
    \bigl(\Phi_p(r^d)-r\bigr)
    \right|_{r=r^*}
    =
    d(r^*)^{d-1}\Phi_p'(x^*)-1
    =
    \phi'(x^*)-1
    =
    h'(x^*).
\]
Thus $h'(x^*)<0$, so $\phi'(x^*)<1$. Therefore $0<\phi'(x^*)<1$.
\end{proof}

\section{Weak spatial mixing under positive-density wired boundaries}
\label{sec:wsm}

Fix $\Delta\ge3$ and write $d:=\Delta-1$. We use $\dary$ for the complete $d$-ary tree of height $h$ with the root $\rho$, and $\dreg$ for the finite $\Delta$-regular tree of height $h$ with the root $\rho$, namely the ball of radius $h$ around a fixed vertex in the infinite $\Delta$-regular tree.

As rooted trees, $\dary$ and $\dreg$ have the same branching structure except at the root: the root of $\dary$ has $d$ children, whereas the root of $\dreg$ has $\Delta=d+1$ children. Every other internal vertex has $d$ children in both trees.

We now turn the stability of the nontrivial fixed point into a spatial-mixing estimate for random boundary conditions whose wired component has positive density.
The fixed-point analysis from Section~\ref{subsec:fixed-points} shows that, when $p>\pc$, the homogeneous recursion $\phi(x)=\Phi(x)^d$ has an unstable fixed point at $1$ and a unique attractive fixed point $x^*>1$. Thus, once the messages are uniformly separated from $1$, their subsequent evolution toward the root is controlled by the attractivity of $x^*$.

An arbitrary single-component boundary condition may not produce such a separation. For instance, if only one leaf belongs to the wired component, then the influence of that leaf is transmitted along a single path and may vanish as the height increases. We therefore work with boundary conditions whose wired component occupies a positive density of the boundary. Although our main application is to the complete trees $\dary$ or the $\Delta$-regular trees $\dreg$, the following notion is defined for an arbitrary finite tree.

\begin{definition}[$\theta$-wired boundary condition]
\label{def:theta-wired}
Let $\theta\in(0,1]$ and $T$ be a finite tree. A distribution over single-component boundary conditions on $\partial T$ is called \emph{$\theta$-wired} if the distribution of the wired component stochastically dominates, with respect to the inclusion order on $2^{\partial T}$, the distribution of a random subset $A\subseteq\partial T$ in which every boundary vertex is included independently with probability $\theta$. Equivalently, there is a coupling such that $A\subseteq\mathcal C_1(\xi)$ almost surely.
\end{definition}

Thus, a $\theta$-wired boundary condition may leave many boundary vertices unwired, but its wired component stochastically contains a product Bernoulli subset of positive density. We now state the main result of this section.

\begin{theorem}[Weak spatial mixing under $\theta$-wired random boundaries]
\label{thm:wsm}
Let $0<q<1$, $\Delta\ge3$, $p>\pc$, and $\theta\in(0,1]$. Let $\mathsf M$ be a $\theta$-wired distribution on single-component boundary conditions on $\partial \dreg$, and let $\xi\sim\mathsf M$.

There exist constants $c>0, C>0, \beta\in(0,1), \gamma>0,$ depending only on $p,q,d,$ and $\theta$, such that the following holds. For every vertex $u\in V(\dreg)$ with $H(u):=\dist\bigl(u,\partial \dreg\bigr)\ge\gamma$, we have 
\[
\mathbb P_{\xi\sim\mathsf M}\Bigg(
\left|
    \mu_{\dreg}^{\mathbf 1}
    \bigl(u\sim\partial \dreg\bigr)
    -
    \mu_{\dreg}^{\xi}
    \bigl(u\sim\mathcal C_1(\xi)\bigr)
\right|
>
C\beta^{H(u)}
\Bigg)
\le
\exp\left(-c d^{\sqrt{H(u)}}\right).
\]
\end{theorem}

The proof has two ingredients. First, a positive density of wired leaves pushes the messages uniformly away from the unstable fixed point $1$. Second, once the messages lie in a compact subset of $(1,\infty)$, the scalar attractivity of $x^*$ gives exponential contraction for arbitrary non-homogeneous message assignments.

\subsection{Amplification away from the trivial fixed point}
\label{subsec:amplification}

We first establish a deterministic estimate relating the density of wired leaves below a vertex to the message at that vertex.

\begin{lemma}
\label{lem:density-amplification}
Let $0<q<1, d\ge2, p>\pc$. Let $\xi$ be a single-component boundary condition on $\partial \dary$. For a vertex $u\in V(\dary)$, let $H(u):=\dist\bigl(u,\partial \dary\bigr)$ and define $\alpha_\xi(u):=
\left|
    \mathcal C_1(\xi)
    \cap
    \partial \dary^u
\right| / d^{H(u)}$. 
Then there exist constants $\eta>0$ and $s\in(0,1)$, depending only on $p,q,$ and $d$, such that $    f^\xi(u)\ge\exp\left(\eta\,\alpha_\xi(u)^s\right)$.
\end{lemma}

\begin{proof}
We first define an auxiliary function $\mathcal L(z):=\log\Phi(e^z)$ for $z\in[0,\infty)$, with $\mathcal L(\infty):=\log\Phi(\infty)=-\log(1-p).$ Since $\Phi(1)=1$, we have $\mathcal L(0)=0$. Differentiating gives
\[
    \mathcal L'(z)
    =
    \frac{\Phi'(e^z)}
    {\Phi(e^z)}
     e^z
    =
    \frac{p\bigl(1+(q-1)(1-p)\bigr)e^z}
    {\bigl(e^z+(q-1)(1-p)\bigr)
     \bigl((1-p)e^z+p+(q-1)(1-p)\bigr)}
    >0.
\]
Thus, $\mathcal L$ is strictly increasing on $[0,\infty)$. Moreover, $\mathcal L'(0)=\Phi'(1)=p/(p+q(1-p))=\widehat p$.

Since the condition $p>\pc$ is equivalent to $d\widehat p>1$, we may choose $\lambda\in\left(1/d,\widehat p\right)$. By continuity of $\mathcal L'$ at $0$, there exists $\eta>0$ such that $\mathcal L(z)\ge\lambda z$ for every $z\in[0,\eta]$. Since $\mathcal L$ is increasing, this implies
\begin{equation}
\label{eq:log-response-lower-bound}
    \mathcal L(z)
    \ge
    \lambda\min\{z,\eta\}
    \qquad
    \text{for every }z\in[0,\infty].
\end{equation}
Set $s:=\frac{\log(1/\lambda)}{\log d}$. Note that $s\in(0,1)$ and $\lambda d^s=1$. Consider the rooted subtree $T^u$ below $u$. For each vertex $v\in V(T^u)$, define
\[
    a(v)
    :=
    \frac{
        \min\{\log f^\xi(v),\eta\}
    }{\eta}.
\]
Thus $a(v)\in[0,1]$. At a boundary vertex $v$,
\[
    a(v)
    =
    \begin{cases}
        1, & v\in\mathcal C_1(\xi),\\
        0, & v\notin\mathcal C_1(\xi).
    \end{cases}
\]
For an internal vertex $v$, the message recursion and \eqref{eq:log-response-lower-bound} give
\[
    a(v)
    =
    \min\left\{
        1,
        \frac1\eta
        \sum_{w\in N(v)}
        \mathcal L\bigl(\log f^\xi(w)\bigr)
    \right\}
    \ge
    \min\left\{
        1,
        \lambda
        \sum_{w\in N(v)}a(w)
    \right\}.
\]

We prove by induction from the leaves toward $u$ that
\begin{equation}
\label{eq:a-alpha-induction}
    a(v)\ge\alpha_\xi(v)^s
\end{equation}
for every $v\in V(\dary^u)$. The claim holds at the leaves because both sides are either $0$ or $1$. Suppose that it holds for the children of an internal vertex $v$. Since $\alpha_\xi(v)=\frac1d\sum\limits_{w\in N(v)}\alpha_\xi(w)$, we obtain
\[
    a(v)
    \ge
    \min\left\{
        1,
        \lambda
        \sum_{w\in N(v)}
        \alpha_\xi(w)^s
    \right\}.
\]
Because $s\in(0,1)$, the function $x\mapsto x^s$ is subadditive on
$[0,\infty)$. Hence
\[
    \sum_{w\in N(v)}
    \alpha_\xi(w)^s
    \ge
    \left(
        \sum_{w\in N(v)}
        \alpha_\xi(w)
    \right)^s.
\]
Using $\lambda d^s=1$, we obtain $a(v)\ge\min\left\{1,\lambda\bigl(d\alpha_\xi(v)\bigr)^s\right\}=\min\{1,\alpha_\xi(v)^s\}=\alpha_\xi(v)^s$. This proves \eqref{eq:a-alpha-induction}.

Applying the bound at $u$ gives $\min\{\log f^\xi(u),\eta\}\ge\eta\alpha_\xi(u)^s$. Since $\alpha_\xi(u)^s\le1$, it follows that $\log f^\xi(u)\ge\eta\alpha_\xi(u)^s$, which proves the result.
\end{proof}

The preceding lemma is the point at which the assumption $q<1$ enters the spatial-mixing argument in a particularly useful form. It replaces the positive-association and stochastic-domination arguments commonly used when $q>1$ in~\cite[Lemma 2.8]{blanca2026uniqueness}. The only required supercriticality condition is $d\cdot \mathcal L'(0) = d\widehat p >1$, which is exactly equivalent to $p>\pc$.

We now apply Lemma~\ref{lem:density-amplification} to a $\theta$-wired random boundary condition.

\begin{lemma}
\label{lem:uniform-positive}
Let $0<q<1, \Delta\ge3, d=\Delta-1, p>\pc$, and $\theta\in(0,1]$. Let $\mathsf M$ be a $\theta$-wired distribution on single-component boundary conditions on $\partial\dreg$, and let $\xi\sim\mathsf M$. Fix $u\in V(\dreg)$ and write $H:=\dist(u,\partial\dreg)$, where $\partial\dreg$ is the original leaf boundary. Let $k,\ell\ge1$ satisfy $k+\ell=H$. View $\dreg$ as rooted at $u$, and for $x\in\partial B_k(u)$, let $f_u^\xi(x)$ denote the message sent toward $u$ from the descendent subtree below $x$, using the same global wired component $\mathcal C_1(\xi)$ restricted to the original boundary vertices in that subtree.

There exist constants $\varepsilon>0$ and $K<\infty$, depending only on $p,q,d,$ and $\theta$, such that, with probability at least 
\[
    1-
    \exp\left(
        -\frac{\theta d^\ell}8
        +\log\Delta
        +(k-1)\log d
    \right),
\]
we have $1+\varepsilon \le f_u^\xi(x) \le K$ for every $x\in\partial B_k(u)$.
\end{lemma}

\begin{proof}
Couple $\xi\sim\mathsf M$ with a Bernoulli subset $A\subseteq\partial\dreg$ as in Definition~\ref{def:theta-wired}, so that $A\subseteq\mathcal C_1(\xi)$. Fix $x\in\partial B_k(u)$, and orient $\dreg$ away from $u$. Let $h$ be the height of $\dreg$ and set $t:=\dist(\rho,u)$, so that $H=h-t$. There are two cases.

Suppose first that $x$ is a descendant of $u$ in the original rooting at $\rho$. Then the subtree below $x$ is a complete $d$-ary tree of height $H-k=\ell$. By the Chernoff bound, except with probability at most $\exp\left(-\frac{\theta d^{\ell}}8\right)$, at least a $\theta/2$ fraction of the leaves of this branch belong to
$A$. On this event, Lemma~\ref{lem:density-amplification} gives
\[
    f_u^\xi(x)
    \ge
    \exp\left(
        \eta\left(\frac\theta2\right)^s
    \right)
    =:
    1+\varepsilon_0
\]
for some constant $\eta > 0$ and $s \in (0, 1)$. Set $\varepsilon:=\min\left\{\varepsilon_0,\Phi(1+\varepsilon_0)-1\right\}>0$.

Suppose next that $x$ is not a descendant of $u$. Let $j\ge1$ be the number of initial steps toward $\rho$ on the path from $u$ to $x$. Since $d\ge2$, $x$ has a child $y$ in the original rooting that is not on the path to $u$. This is also a child of $x$ in the rooting away from $u$, and its subtree is a complete $d$-ary tree of height $h-(t+k-2j+1)=\ell+2j-1\ge\ell$. Therefore, except with probability at most $\exp\left(-\frac{\theta d^\ell}{8}\right)$, we have the same estimate $f_u^\xi(y)\ge1+\varepsilon_0$. Since every other child message entering $x$ is at least $1$, the message recursion gives
\[
    f_u^\xi(x)
    \ge
    \Phi(1+\varepsilon_0)\Phi(1)^{d-1}
    =
    \Phi(1+\varepsilon_0)
    \ge
    1+\varepsilon.
\]

Since $|\partial B_k(u)|=\Delta d^{k-1}$, a union bound gives the stated probability. Finally, every $x\in\partial B_k(u)$ has $d$ children. Hence $f_u^\xi(x)\le\Phi(\infty)^d=(1-p)^{-d}$. The result follows with $K=(1-p)^{-d}$.
\end{proof}

\subsection{Contraction inside the positive basin}
\label{subsec:contraction}

Once all messages lie in a compact interval contained in $(1,\infty)$, the non-homogeneous recursion can be controlled using the homogeneous map $\phi$.

\begin{lemma}
\label{lem:uniform-contraction}
Let $0<q<1$, $d\ge2$, and $p>\pc$. Fix constants $\varepsilon>0$ and $K<\infty$. Consider the complete $d$-ary tree $\mathcal T_k$ of height $k$, and let $g_1,g_2:\partial\mathcal T_k\to[1+\varepsilon,K]$ be two boundary message assignments. Let $\widehat f_1$ and $\widehat f_2$ be the message functions obtained by propagating $g_1$ and $g_2$ toward the root using
\[
    \widehat f_i(u)
    =
    \prod_{w\in N(u)}
    \Phi\bigl(\widehat f_i(w)\bigr).
\]
There exist constants $C_0<\infty$ and $\beta\in(0,1)$, depending only on $p,q,d,\varepsilon,$ and $K$, such that for every $u\in V(\mathcal T_k)$,
\[
    \left|
        \widehat f_1(u)-\widehat f_2(u)
    \right|
    \le
    C_0
        \beta^{\dist(u,\partial\mathcal T_k)}.
\]

Moreover, let $\Delta=d+1$. The same estimate holds for the finite $\Delta$-regular tree $\widehat{\mathcal T}_k$ of height $k$ with boundary message assignments $g_1,g_2:\partial\widehat{\mathcal T}_k\to[1+\varepsilon,K]$.
\end{lemma}

\begin{proof}
Define two scalar sequences $(a_k)_{k\ge0}$ and $(b_k)_{k\ge0}$ by $a_{t+1}=\phi(a_t)$, $b_{t+1}=\phi(b_t)$ with $a_0=1+\varepsilon$, $b_0=K$. Since $\Phi$ is increasing, an induction from the boundary shows that every propagated message at distance $t$ from the boundary lies in $[a_t,b_t]$. In particular, for every vertex $u\in V(\mathcal T_k)$ with $\dist(u,\partial\mathcal T_k)=t$, $a_t\le\widehat f_i(u)\le b_t$ for $i\in\{1,2\}$, and therefore $\left|\widehat f_1(u)-\widehat f_2(u)\right|\le b_t-a_t$.

By Proposition~\ref{prop:attractivity}, we have $a_t\to x^*, b_t\to x^*$ as $t\to\infty$. Moreover, $0<\phi'(x^*)<1$. Choose $\delta>0$ sufficiently small that
\[
    \beta
    :=
    \sup_{x\in[x^*-\delta,x^*+\delta]}
    \phi'(x)
    <1.
\]
Since both sequences converge to $x^*$, there exists $t_0<\infty$ such that $a_t,b_t \in [x^*-\delta,x^*+\delta]$ for every $t\ge t_0$.

For every $t\ge t_0$, the mean value theorem gives $b_{t+1}-a_{t+1}=\phi(b_t)-\phi(a_t)\le\beta\bigl(b_t-a_t\bigr)$. Consequently,
\[
    b_t-a_t
    \le
    (b_{t_0}-a_{t_0})\beta^{t-t_0}
    \qquad
    \text{for every }t\ge t_0.
\]
Thus $b_t-a_t\le C_0\beta^t$ for every $t\ge t_0$, where $C_0:=(b_{t_0}-a_{t_0})\beta^{-t_0}$. Since there are only finitely many indices $0\le t<t_0$, enlarging $C_0$ if necessary makes the same bound valid for every $t\ge0$.

The finite $\Delta$-regular tree $\widehat{\mathcal T}_k$ differs from the complete $d$-ary tree only at the root $\rho$. Therefore, it suffices to prove the estimate separately at $\rho$. Each of the $\Delta$ branches below the root is a complete $d$-ary tree of height $k-1$. Thus, if $v_1,\ldots,v_\Delta$ are the children of $\rho$, then
\[
    \left|
        \widehat f_1(v_i)-\widehat f_2(v_i)
    \right|
    \le
    C_0\beta^{k-1}
\]
for every $i\in\{1,\ldots,\Delta\}$.

The root update is given by $\Phi_\Delta(x_1,\ldots,x_\Delta):=\prod_{i=1}^{\Delta}\Phi(x_i)$. Since all relevant messages lie in a fixed compact interval $\left[1, \max\{K, (1-p)^{-d}\}\right]$, $\Phi_\Delta$ is uniformly Lipschitz on the corresponding product
space. Therefore,
\[
    \left|
        \widehat f_1(\rho)-\widehat f_2(\rho)
    \right|
    \le
    L_\Delta
    \sum_{i=1}^{\Delta}
    \left|
        \widehat f_1(v_i)-\widehat f_2(v_i)
    \right|
    \le
    L_\Delta\Delta C_0\beta^{k-1}
\]
where $L_\Delta<\infty$ is a Lipschitz constant of $\Phi_\Delta$. The result follows after absorbing the additional constants and the factor $\beta^{-1}$ into $C_0$.
\end{proof}

\subsection{Proof of weak spatial mixing}
\label{subsec:proof-wsm}

\begin{proof}[Proof of Theorem~\ref{thm:wsm}]
Fix a vertex $u\in V(\dreg)$, and write $H:=H(u)=\dist\bigl(u,\partial\dreg\bigr)$. We now view the entire tree $\dreg$ as rooted at $u$ and denote the tree by $\dreg^{(u)}$. Set $\ell:=\left\lceil\sqrt H\right\rceil$ and $k:=H-\ell$. 

We take $\gamma\ge4$, so that $k,\ell\ge1$ whenever $H\ge\gamma$. Since $k<H$, the ball $B_k(u)$ does not meet $\partial\dreg$. Consequently, when rooted at $u$, the ball $B_k(u)$ is exactly the finite $\Delta$-regular tree of height $k$: its root has $\Delta$ children and every other internal vertex has $d=\Delta-1$ children. For the remainder of the proof, write $f_u^\xi(v)$ for the message computed after viewing the entire tree $\dreg$ as rooted at $u$. 

By Lemma~\ref{lem:uniform-positive}, with probability at least $1-\exp\left(-\frac{\theta d^\ell}{8}+\log \Delta + (k-1)\log d\right)$, the boundary messages entering $B_k(u)$ satisfy $1+\varepsilon \le f^\xi(x) \le K$ with some constant $\varepsilon>0$ and $K<\infty$ for every $x\in\partial B_k(u)$. The same proof, with all-wired, gives these bounds deterministically under the all-wired boundary condition.

We apply Lemma~\ref{lem:uniform-contraction} to $B_k(u)$, using the two boundary message assignments $\bigl(f_u^\xi(x)\bigr)_{x\in\partial B_k(u)}$ and $\bigl(f_u^{\mathbf 1}(x)\bigr)_{x\in\partial B_k(u)}$. It follows that
\[
    \left|
        f_u^\xi(u)-f_u^{\mathbf 1}(u)
    \right|
    \le
    C_0\beta_0^k
\]
for some constants $C_0<\infty$ and $\beta_0\in(0,1)$.

Therefore Fact~\ref{fact:message-probability} gives
\[
\left|
    \mu_{\dreg}^{\mathbf 1}
    \bigl(u\sim\partial\dreg\bigr)
    -
    \mu_{\dreg}^{\xi}
    \bigl(u\sim\mathcal C_1(\xi)\bigr)
\right|
\le
\frac1q
\left|
    f_u^{\mathbf 1}(u)-f_u^\xi(u)
\right|
\le
\frac{C_0}{q}\beta_0^k.
\]

Since $k=H-\lceil\sqrt H\rceil$, we may choose a constant $\beta\in(\beta_0,1)$ and then choose $\gamma$ sufficiently large so that $\beta_0^{H-\lceil\sqrt H\rceil} \le\beta^H$ for every $H\ge\gamma$. After absorbing the factor $1/q$ into the
constant, we obtain
\[
\left|
    \mu_{\dreg}^{\mathbf 1}
    \bigl(u\sim\partial\dreg\bigr)
    -
    \mu_{\dreg}^{\xi}
    \bigl(u\sim\mathcal C_1(\xi)\bigr)
\right|
\le
C\beta^H.
\]

Finally, for all sufficiently large $H$, $\theta d^{\lceil\sqrt H\rceil}/8-\log\Delta-\bigl(H-\lceil\sqrt H\rceil-1\bigr)\log d\ge c d^{\sqrt H}$ for some constant $c>0$. Thus the preceding estimate fails with
probability at most $\exp\left(-c d^{\sqrt H}\right)$. This proves the theorem.
\end{proof}

\begin{remark}
\label{rem:wsm-rate}
Let $\beta^*:=\phi'(x^*)\in(0,1)$. The contraction rate in Lemma~\ref{lem:uniform-contraction} may be chosen arbitrarily close to $\beta^*$. More precisely, for every $\delta\in(0,1-\beta^*)$, there exists $C_\delta<\infty$ such that, on the good event used in the proof of Theorem~\ref{thm:wsm},
\[
\left|
    \mu_{\dreg}^{\mathbf 1}
    \bigl(u\sim\partial\dreg\bigr)
    -
    \mu_{\dreg}^{\xi}
    \bigl(u\sim\mathcal C_1(\xi)\bigr)
\right|
\le
C_\delta
\bigl(\beta^*+\delta\bigr)^{
    H(u)-\lceil\sqrt{H(u)}\rceil
}.
\]
For the weak spatial mixing statement, we use the simpler bound
$C\beta^{H(u)}$ with a fixed $\beta\in(\beta^*,1)$.
\end{remark}

To prove uniqueness in the following section, we need a local version of Theorem~\ref{thm:wsm}. Although Theorem~\ref{thm:wsm} is stated for connection probabilities at a single vertex, the message estimate in its proof controls every fixed finite-dimensional marginal.

\begin{corollary}[Weak spatial mixing for local marginals]
\label{cor:local-wsm}
Let $0<q<1$, $\Delta\ge3$, $p>\pc$, and $\theta\in(0,1]$. Let $\mathsf M$ be a $\theta$-wired distribution on single-component boundary conditions on $\partial\dreg$, and let $\xi\sim\mathsf M$. Fix a finite edge set $\Lambda\subset E(\mathcal T_\infty)$, and suppose that $\Lambda\subset E(\dreg)$. Write $r_h:=\dist(\Lambda,\partial\dreg)$.
There exist constants $A_\Lambda,C_\Lambda>0$, $c>0$, $\beta\in(0,1),$ and $\gamma>0$, depending only on $p,q,d,\theta,$ and $\Lambda$, such that, whenever $r_h\ge\gamma$,
\[
\mathbb P_{\xi\sim\mathsf M}
\left(
    \left\|
        \mu_{\dreg}^{\xi}\vert_\Lambda
        -
        \mu_{\dreg}^{\mathbf 1}\vert_\Lambda
    \right\|_{\mathrm{TV}}
    >
    C_\Lambda\beta^{r_h}
\right)
\le
A_\Lambda
\exp\left(
    -c d^{\sqrt{r_h}}
\right).
\]
Consequently,
\[
\mathbb E_{\xi\sim\mathsf M}
\left[
    \left\|
        \mu_{\dreg}^{\xi}\vert_\Lambda
        -
        \mu_{\dreg}^{\mathbf 1}\vert_\Lambda
    \right\|_{\mathrm{TV}}
\right]
\le
C_\Lambda\beta^{r_h}
+
A_\Lambda
\exp\left(
    -c d^{\sqrt{r_h}}
\right).
\]
\end{corollary}

\begin{proof}
Let $T_\Lambda$ be the minimal connected subtree of $\dreg$ containing the root $\rho$ and every endpoint of every edge in $\Lambda$. The distribution of the configuration on $E(T_\Lambda)$ depends on the boundary condition outside $T_\Lambda$ only through the finitely many messages entering $T_\Lambda$ from the components of $\dreg\setminus T_\Lambda$. Enumerate these components by $T'_1,\ldots,T'_m$, and let $u_i$ be the root of $T'_i$ adjacent to $T_\Lambda$.

Since $T_\Lambda$ contains $\rho$ and is ancestor-closed, each $T'_i$ is a complete $d$-ary tree. Let $h_i$ denote its height. Note that $h_i\ge r_h-1$ for every $i\in[m]$. For each $i$, set $\ell_i:=\left\lceil\sqrt{h_i}\right\rceil$ and $k_i:=h_i-\ell_i$. 

Applying Lemma~\ref{lem:uniform-positive} to $T'_i$, with $h_i=k_i+\ell_i$, shows that, except on an event of probability at most $\exp(-c d^{\sqrt{h_i}})$, all messages entering the upper subtree of height $k_i$ lie in a fixed interval $[1+\varepsilon,K]$; the corresponding bounds under the all-wired boundary condition hold deterministically. Lemma~\ref{lem:uniform-contraction} then gives
\[
    \left|
        f^\xi(u_i)-f^{\mathbf 1}(u_i)
    \right|
    \le
    C\beta_0^{k_i}
\]
for constants $C<\infty$ and $\beta_0\in(0,1)$ independent of $i$ and $h$. Since $m$ is fixed, a union bound over the components, followed by choosing any $\beta\in(\beta_0,1)$ and adjusting the constants, shows that, except on an event of probability at most $A_\Lambda\exp\left(-c d^{\sqrt{r_h}}\right)$, the incoming messages under $\xi$ and under $\mathbf 1$ satisfy
\[
    \max_{1\le i\le m}
    \left|
        f^\xi(u_i)-f^{\mathbf 1}(u_i)
    \right|
    \le
    C_\Lambda\beta^{r_h}.
\]

On the same event, all relevant messages lie in a fixed compact interval contained in $(1,\infty)$. We now make explicit why the individual scalar messages determine the joint law on $T_\Lambda$, even though the incoming branches share the same wired component. For branch $i$, let $Z_i^0$ and $Z_i^1$ be its restricted partition functions according to whether its attachment vertex is disconnected from or connected to the common wired component, and write
\[
    x_i
    :=
    1+q\frac{Z_i^1}{Z_i^0},
    \qquad
    y_i
    :=
    \frac{Z_i^1}{Z_i^0}
    =
    \frac{x_i-1}{q}.
\]

For a fixed configuration $\omega$ on $E(T_\Lambda)$, summing over all configurations in the incoming branches gives an unnormalized weight of the form
\[
    \widetilde W(\omega;\mathbf y)
    =
    \left(\prod_{i=1}^m Z_i^0\right)
    \sum_{S\subseteq[m]}
    c(\omega,S)
    \prod_{i\in S}y_i.
\]
Indeed, conditional on the set $S$ of branches whose attachment vertices connect to the common wired component, the sums over the branch configurations factor as
$\prod_{i\notin S}Z_i^0\prod_{i\in S}Z_i^1$.

On the good event, $x_i>1$ for every $i$, so each incoming branch has a nonempty wired component. This component is present and counted once in both restricted sums $Z_i^0$ and $Z_i^1$. The product of the restricted partition functions therefore counts the common wired component once in each isolated branch. 

Gluing the branches replaces these $m$ copies by one and contributes the constant factor $q^{1-m}$, independent of $\omega$ and $S$. This factor is absorbed into the overall normalization. The remaining gluing correction depends only on $\omega$ and $S$, and together with the edge weight of $\omega$ defines $c(\omega,S)\ge0$. Thus the dependence between different incoming branches is contained in the finite coefficients $c(\omega,S)$, whereas each individual branch enters only through its scalar ratio $y_i$. Denote the resulting law on $E(T_\Lambda)$ by $\mu_{T_\Lambda}^{\mathbf y}$. The common factor $\prod_i Z_i^0$ cancels after normalization, so
\[
    \mu_{T_\Lambda}^{\mathbf y}(\omega)
    =
    \frac{
        \displaystyle
        \sum_{S\subseteq[m]}
        c(\omega,S)
        \prod_{i\in S}y_i
    }{
        \displaystyle
        \sum_{\omega'}
        \sum_{S\subseteq[m]}
        c(\omega',S)
        \prod_{i\in S}y_i
    }.
\]

Since $T_\Lambda$ is fixed, this is a rational function of the finitely many incoming messages $x_1,\ldots,x_m$. Its denominator is continuous and strictly positive on the fixed compact message domain, and is therefore uniformly bounded away from zero. Hence these rational functions are uniformly Lipschitz there. It follows that
\[
    \left\|
        \mu_{\dreg}^{\xi}\vert_\Lambda
        -
        \mu_{\dreg}^{\mathbf 1}\vert_\Lambda
    \right\|_{\mathrm{TV}}
    \le
    C_\Lambda\beta^{r_h}
\]
on the good event. The first claim follows after adjusting the constants. The expectation bound follows because total variation distance is at most $1$.
\end{proof}

\section{Uniqueness of the infinite-volume wired tree measure}
\label{sec:uniqueness}

In this section we prove the new supercritical uniqueness statement and give a self-contained account of the full phase diagram. We use the DLR formalism from Definition~\ref{def:dlr}, under which all infinite connected components of the exterior configuration are identified with one another at infinity. For $p>\pc$, uniqueness follows from weak spatial mixing. For $p\le\pc$, the classical wired measure is an explicit Bernoulli product measure; we include a direct proof for completeness.

\subsection{The supercritical regime \texorpdfstring{$p>\pc$}{p > pc}}

We use the weak spatial mixing estimate from Theorem~\ref{thm:wsm} to prove uniqueness in the supercritical regime. The existence of the all-wired weak limit is known in greater generality from~\cite[Theorem~10.74]{grimmett2006randomcluster}. We nevertheless include a short proof for balls in the present message coordinates, since it identifies the limiting boundary law used below. For each $h\ge1$, let $\dreg$ denote the finite $\Delta$-regular tree of height $h$, and let $\mu_{\dreg}^{\mathbf 1}$ be the corresponding random-cluster measure with the all-wired boundary condition.

\begin{lemma}[Existence of the wired finite-volume limit]
\label{lem:wired-limit}
Let $0<q<1, \Delta\ge3$ and $p>\pc$. The sequence of the random-cluster measures with the all-wired boundary condition $\bigl(\mu_{\dreg}^{\mathbf 1}\bigr)_{h\ge1}$ converges locally as $h\to\infty$. We denote its limit by
\[
    \mu^*
    :=
    \lim_{h\to\infty}
    \mu_{\dreg}^{\mathbf 1}.
\]
\end{lemma}

\begin{proof}
Fix a finite edge set $\Lambda\subset E(\mathcal T_\infty)$, and choose $r\ge1$ such that $\Lambda\subset E(B_r(\rho))$. Write $E_r:=E(B_r(\rho))$. Under the all-wired boundary condition on a tree of height $h>r$, each component of $\dreg\setminus B_r(\rho)$ sends the homogeneous message $x_{h-r}=\phi^{(h-r)}(\infty)$ toward $B_r(\rho)$. By Proposition~\ref{prop:attractivity}, $x_{h-r}\to x^*$ as $h\to\infty$.

We make explicit how these incoming messages determine the marginal on $B_r(\rho)$. Let $m$ be the number of components of $\dreg\setminus B_r(\rho)$ attached to $B_r(\rho)$, and denote by $Z^0$ and $Z^1$ the restricted partition functions of any one such component according to whether its attachment vertex is disconnected from or connected to the wired boundary component. By the definition of the message in Definition~\ref{def:message}, $x_{h-r}=1+q Z^1/Z^0$.
Let 
\[
    y_{h-r}=\frac{(x_{h-r}-1)}{q}=\frac{Z^1}{Z^0} > 0.
\] 
Fix a configuration $\omega\in\{0,1\}^{E_r}$. After summing over all configurations outside $B_r(\rho)$, its unnormalized marginal weight can be written as
\[
    \widetilde W_h(\omega)
    =
    (Z^0)^m
    \sum_{S\subseteq[m]}
    c(\omega,S)y_{h-r}^{|S|}.
\]
Here $S$ records the exterior components whose attachment vertices are connected to the wired boundary component. Conditional on $S$, the exterior sums factor as $(Z^0)^{m-|S|}(Z^1)^{|S|}$.

Under the all-wired boundary condition, every incoming branch contains a nonempty copy of the common wired component. This component is counted once in both $Z^0$ and $Z^1$, whether or not the attachment vertex connects to it.

Identifying the $m$ copies of the common wired component therefore contributes the constant factor $q^{1-m}$, independent of $S$. Every remaining cluster-count correction depends only on $B_r(\rho)$, $\omega$, and $S$. Absorbing these corrections and the edge weight of $\omega$ into $c(\omega,S)\ge0$ shows in particular that these coefficients do not depend on $h$. The common factor $(Z^0)^m$ cancels upon normalization. Hence, for every $\eta\in\{0,1\}^{\Lambda}$,
\[
    \mu_{\dreg}^{\mathbf 1}
    (\omega_\Lambda=\eta)
    =
    \frac{
        \displaystyle
        \sum_{\substack{
            \omega\in\{0,1\}^{E_r}:\\
            \omega_\Lambda=\eta
        }}
        \sum_{S\subseteq[m]}
        c(\omega,S)y_{h-r}^{|S|}
    }{
        \displaystyle
        \sum_{\omega\in\{0,1\}^{E_r}}
        \sum_{S\subseteq[m]}
        c(\omega,S)y_{h-r}^{|S|}
    }.
\]
Both sums are finite. Since $x_{h-r}\to x^*>1$, $y_{h-r}$ also converges to $y^*$ for some $y^*>0$. Thus every term in the numerator and denominator converges. The limiting denominator is strictly positive, and therefore $\mu_{\dreg}^{\mathbf 1}(\omega_\Lambda=\eta)$ converges for every $\eta$. Since $\{0,1\}^{\Lambda}$ is finite, this proves convergence of $\mu_{\dreg}^{\mathbf 1}\vert_\Lambda$ in total variation.

The limiting marginal is obtained from the same finite sum by replacing $x_{h-r}$ with $x^*$. For each $r$, let $\mu_r^*$ be the limiting law on $B_r(\rho)$. If $s>r$, then for every $h>s$,
\[
    \left(
        \mu_{\dreg}^{\mathbf 1}\vert_{E_s}
    \right)\vert_{E_r}
    =
    \mu_{\dreg}^{\mathbf 1}\vert_{E_r}.
\]
Passing to the limit as $h\to\infty$ on these finite configuration spaces gives $\mu_s^*\vert_{E_r}=\mu_r^*$. Thus the limiting marginals are projectively consistent, and the Kolmogorov extension theorem gives a probability measure $\mu^*$ on $\{0,1\}^{E(\mathcal T_\infty)}$ with these marginals.
\end{proof}

\begin{lemma}[DLR property of the wired limit]
\label{lem:wired-limit-dlr}
The measure $\mu^*$ from Lemma~\ref{lem:wired-limit} satisfies the wired DLR specification of Definition~\ref{def:dlr}.
\end{lemma}

\begin{proof}
This is the all-wired-limit identification in~\cite[Theorem~10.82(c)]{grimmett2006randomcluster}, with the specification given in~\cite[Equation~(10.66)]{grimmett2006randomcluster}. The argument is presented for the binary tree, but it uses only the finite decomposition into descendant branches and applies unchanged with any fixed number $d$ of such branches. Lemma~\ref{lem:wired-limit} identifies $\mu^*$ as the local limit of the same all-wired finite-volume measures, so the cited identification applies.
\end{proof}

The positivity of the supercritical percolation probability and the location of the tree threshold are consequences of~\cite[Proposition~10.81 and Theorem~10.82(a),(b)]{grimmett2006randomcluster}. The next proposition records a convenient explicit lower bound in terms of the fixed point $x^*$.

\begin{proposition}[Percolation of the supercritical wired limit]
\label{prop:wired-limit-percolates}
Let $0<q<1$, $\Delta\ge3$, and $p>\pc$. Let $x^*>1$ be the nontrivial fixed point of the homogeneous recursion, and define $\widehat x^*:=\Phi(x^*)^\Delta$. Then the all-wired finite-volume limit from Lemma~\ref{lem:wired-limit} satisfies
\[
    \mu^*(\rho\leftrightarrow\infty)
    \ge
    \frac{\widehat x^*-1}{\widehat x^*+q-1}
    >0.
\]
\end{proposition}

\begin{proof}
Let $\widehat{\mathcal T}_h$ be the ball of radius $h$ centered at $\rho$, with all vertices at distance $h$ wired together. Each of the $\Delta$ branches incident to the root is a complete $d$-ary tree of height $h-1$, where $d=\Delta-1$. 

If $x_{h-1}:=\phi^{(h-1)}(\infty)$ is the message sent from each such branch, then the recursion at the root gives $\widehat x_h:=\Phi(x_{h-1})^\Delta$. 
By Fact~\ref{fact:message-probability},
\[
    \mu_{\widehat{\mathcal T}_h}^{\mathbf 1}
    \bigl(\rho\leftrightarrow\partial\widehat{\mathcal T}_h\bigr)
    =
    \frac{\widehat x_h-1}{\widehat x_h+q-1}.
\]
Proposition~\ref{prop:attractivity} gives $x_{h-1}\to x^*$ as $h\to\infty$, and hence $\widehat x_h\longrightarrow\widehat x^*=\Phi(x^*)^\Delta$ as $h\to\infty$. Since $x^*>1$, we have $\Phi(x^*)>1$ and therefore $\widehat x^*>1$. It follows that
\begin{equation}
\label{eq:finite-wired-root-connection-limit}
    \lim_{h\to\infty}
    \mu_{\widehat{\mathcal T}_h}^{\mathbf 1}
    \bigl(\rho\leftrightarrow\partial\widehat{\mathcal T}_h\bigr)
    =
    \frac{\widehat x^*-1}{\widehat x^*+q-1}
    >0.
\end{equation}

For a fixed $H\ge1$, regard $\widehat{\mathcal T}_H$ as the subball of $\dreg$ of radius $H$ centered at $\rho$ and let
\[
    A_H
    :=
    \left\{
        \rho\text{ is connected to }\partial\widehat{\mathcal T}_H
        \text{ by an open path contained in }\widehat{\mathcal T}_H
    \right\}.
\]
For every $h\ge H$, a connection from $\rho$ to the wired boundary of $\widehat{\mathcal T}_h$ must cross $\partial\widehat{\mathcal T}_H$, therefore we have
$
    \mu_{\widehat{\mathcal T}_h}^{\mathbf 1}(A_H)
    \ge
    \mu_{\widehat{\mathcal T}_h}^{\mathbf 1}
    \bigl(\rho\leftrightarrow\partial\widehat{\mathcal T}_h\bigr).
$
The event $A_H$ depends on finitely many edges. Local convergence from Lemma~\ref{lem:wired-limit} and~\eqref{eq:finite-wired-root-connection-limit} therefore imply
\[
    \mu^*(A_H)
    \ge
    \frac{\widehat x^*-1}{\widehat x^*+q-1}
    \qquad
    \text{for every }H\ge1.
\]
The events $(A_H)_{H\ge1}$ decrease to $\{\rho\leftrightarrow\infty\}$. Continuity from above now gives
\[
    \mu^*(\rho\leftrightarrow\infty)
    =
    \lim_{H\to\infty}\mu^*(A_H)
    \ge
    \frac{\widehat x^*-1}{\widehat x^*+q-1}
    >0.
\]
\end{proof}

To compare the all-wired limit with an arbitrary candidate wired DLR measure $\nu$ on $\mathcal T_\infty$, consider the boundary condition on $\partial\dreg$ induced by the exterior configuration, with all infinite exterior components wired together. To apply Theorem~\ref{thm:wsm}, we must show that this boundary condition contains wired vertices with a density bounded uniformly in both the radius and the measure. The following bridge lemma provides exactly this input; its proof is given in Section~\ref{subsec:exterior-induced-boundary}.

\begin{restatable}[Exterior-induced boundary conditions]{lemma}{inducedThetaWired}
\label{lem:induced-theta-wired}
Let $0<q<1$, $\Delta\ge3$, and $p>\pc$. There exists $\theta\in(0,1]$, depending only on $p,q,$ and $d$, such that, for every infinite-volume measure $\nu$ satisfying the wired-tree DLR condition and every $h\ge1$, the distribution of $\xi_h(\sigma)$ under $\sigma\sim\nu$ is $\theta$-wired.
\end{restatable}

We now have both ingredients needed to close the supercritical argument. Every wired DLR measure induces a uniformly $\theta$-wired boundary law on distant balls, and the local weak-spatial-mixing estimate erases the effect of replacing that law by the all-wired boundary condition. Hence all of its finite-dimensional marginals must coincide with those of $\mu^*$.

\begin{theorem}[Supercritical wired-DLR uniqueness]
\label{thm:uniqueness}
Let $0<q<1$, $\Delta\ge3$, and $p>\pc$. There exists a unique infinite-volume random-cluster measure satisfying the wired DLR specification of Definition~\ref{def:dlr}, and this measure is the finite-volume limit $\mu^*$ from Lemma~\ref{lem:wired-limit}.
\end{theorem}

\begin{proof}
Existence follows from Lemma~\ref{lem:wired-limit-dlr}. It remains to prove uniqueness.

Let $\nu$ be an arbitrary infinite-volume random-cluster measure satisfying the wired-tree DLR condition. We show that $\nu$ and $\mu^*$ have the same finite-dimensional marginals. 
Fix a finite edge set $\Lambda\subset E(\mathcal T_\infty)$. Let $h$ be sufficiently large that $\Lambda\subset E(\dreg)$, where $\dreg$ is the finite $\Delta$-regular tree of height $h$, and write $r_h:=\dist(\Lambda,\partial\dreg)$. Then $r_h\to\infty$ as $h\to\infty$. 

Let $\sigma\sim\nu$, and let $\xi_h(\sigma)$ be the boundary condition induced on $\partial\dreg$ by the exterior configuration $\sigma\vert_{E(\mathcal T_\infty)\setminus E(\dreg)}$, with all infinite exterior components wired together. By the DLR condition,
\[
    \nu\vert_\Lambda
    =
    \mathbb E_{\sigma\sim\nu}
    \left[
        \mu_{\dreg}^{\xi_h(\sigma)}
        \vert_\Lambda
    \right].
\]
By Lemma~\ref{lem:induced-theta-wired}, the distribution of $\xi_h(\sigma)$ is $\theta$-wired for some $\theta\in(0,1]$ independent of $h$. By convexity of total variation and Corollary~\ref{cor:local-wsm},
\[
    \left\|
        \nu\vert_\Lambda
        -
        \mu_{\dreg}^{\mathbf 1}\vert_\Lambda
    \right\|_{\mathrm{TV}}
    \le
    \mathbb E_{\sigma\sim\nu}
    \left[
        \left\|
            \mu_{\dreg}^{\xi_h(\sigma)}\vert_\Lambda
            -
            \mu_{\dreg}^{\mathbf 1}\vert_\Lambda
        \right\|_{\mathrm{TV}}
    \right]
    \le
    C_\Lambda\beta^{r_h}
    +
    A_\Lambda
    \exp\left(
        -c d^{\sqrt{r_h}}
    \right)
\]
for some constants $A_\Lambda,C_\Lambda>0$, $c>0$, $\beta\in(0,1)$. Letting $h\to\infty$, the right-hand side tends to zero. By Lemma~\ref{lem:wired-limit}, we have $\mu_{\dreg}^{\mathbf 1}\vert_\Lambda\longrightarrow\mu^*\vert_\Lambda$. Hence $\nu\vert_\Lambda=\mu^*\vert_\Lambda$.

Since $\Lambda$ was arbitrary, $\nu$ and $\mu^*$ agree on every finite-dimensional cylinder event. Therefore every wired DLR measure equals $\mu^*$, proving uniqueness.
\end{proof}

\subsubsection{Exterior-induced boundary conditions}
\label{subsec:exterior-induced-boundary}

We now verify that the boundary condition induced by an arbitrary infinite-volume wired Gibbs measure satisfies the hypothesis of Theorem~\ref{thm:wsm}. The main observation is that, whenever none of the children of a vertex is connected to infinity through its own descendant branch, the edges from the vertex to its children are conditionally independent Bernoulli variables with parameter
\[
    \widehat p
    :=
    \frac{p}{p+q(1-p)}.
\]
Since $p>\pc$ is equivalent to $d\widehat p>1$, this conditional product structure generates a supercritical branching mechanism unless the explored open cluster encounters a child that is already connected to infinity. The argument below shows that such a child can be found within a uniformly bounded distance with positive probability. Since every edge has conditional open probability at least $p$, the path from the root to that child is open with uniformly positive probability, and the root is then connected to infinity.

This use of infinite open descendant subtrees has a precursor in the low-temperature uniqueness arguments of Jonasson and Grimmett--Janson, summarized in~\cite[Theorem~10.98]{grimmett2006randomcluster}; see also~\cite{grimmettJanson2005branching,jonasson1999general}. On the binary tree, that $q\ge1$ argument obtains uniqueness under the sufficient condition $p\ge 2q/(2q+1)$, equivalently $\widehat p\ge2/3$, by constructing a random set of vertices that intersects every infinite ray and whose vertices root infinite open subtrees. The role of Lemma~\ref{lem:uniform-branch-survival} below is analogous, but in the $q<1$ setting its probe exploration yields a lower bound uniform under arbitrary exterior conditioning and works throughout the sharp supercritical regime $d\widehat p>1$. On the binary tree, the corresponding condition is $\widehat p>1/2$, rather than the earlier sufficient bound $\widehat p\ge2/3$.

Let $\nu$ satisfy the wired-tree DLR condition and $\sigma\sim\nu$. Let $\xi_h(\sigma)$ be the boundary condition induced on $\partial\dreg$ by the configuration outside $\dreg$, with all infinite exterior components wired together. Explicitly, its wired component is
\[
    \mathcal C_1\bigl(\xi_h(\sigma)\bigr)
    :=
    \left\{
        x\in\partial\dreg:
        x\text{ is connected to infinity through its exterior branch}
    \right\}.
\]
Since the exterior components attached to distinct vertices of $\partial\dreg$ are disjoint infinite $d$-ary branches, $\xi_h(\sigma)$ is a single-component boundary condition.

Fix an orientation of the infinite $\Delta$-regular tree away from a root. For a vertex $u$, let $\mathcal T_u^{\mathrm{ext}}$ denote the infinite $d$-ary descendant branch rooted at $u$, and define
\[
    \mathcal B_u
    :=
    \left\{
        u\text{ is connected to infinity by an open path contained in }
        \mathcal T_u^{\mathrm{ext}}
    \right\}.
\]
We call $u$ \emph{blue} when $\mathcal B_u$ occurs. Let $\mathcal F_u^{\mathrm{ext}}:=\sigma\left(\omega(e):e\notin E(\mathcal T_u^{\mathrm{ext}})\right)$. Thus $\mathcal F_u^{\mathrm{ext}}$ records the entire configuration outside the descendant branch rooted at $u$.

We first isolate the conditional product structure that drives the exploration. Let $u$ be a vertex with children $u_1,\ldots,u_d$, and write
\[
    e_i:=\{u,u_i\},
    \qquad
    \mathcal R_u
    :=
    \bigcup_{i=1}^d\mathcal B_{u_i}.
\]
The event $\mathcal R_u$ means that at least one child of $u$ is blue.

\begin{lemma}
\label{lem:conditional-star-factorization}
Let $0<q<1$, and let $\nu$ be an infinite-volume measure satisfying the wired-tree DLR condition. Fix a vertex $u$ with children $u_1,\ldots,u_d$, and write $e_i:=\{u,u_i\}$. Let $\mathcal H_u := \sigma\left(\omega(e):e\notin\{e_1,\ldots,e_d\}\right)$. Thus $\mathcal H_u$ fixes every edge except the $d$ edges from $u$ to
its children. On $\mathcal R_u^c$, their conditional distribution given
$\mathcal H_u$ is the product Bernoulli distribution with parameter
$\widehat p$. More precisely, for every
$(a_1,\ldots,a_d)\in\{0,1\}^d$,
\[
    \mathbf 1_{\mathcal R_u^c}
    \nu\left(
        \omega(e_i)=a_i,\ i\in[d]
        \,\middle|\,
        \mathcal H_u
    \right)
    =
    \mathbf 1_{\mathcal R_u^c}
    \prod_{i=1}^d
    \widehat p^{a_i}
    (1-\widehat p)^{1-a_i}
\]
almost surely.
\end{lemma}

\begin{proof} Each event $\mathcal B_{u_i}$ depends only on the edges in $\mathcal T_{u_i}^{\mathrm{ext}}$ and is therefore $\mathcal H_u$-measurable. In particular, $\mathcal R_u^c$ is $\mathcal H_u$-measurable.

Fix an exterior configuration $\eta$ of all edges other than $e_1,\ldots,e_d$ for which $\mathcal R_u^c$ occurs. Since $u_i$ is not blue, the open component containing $u_i$ in $\mathcal T_{u_i}^{\mathrm{ext}}$ is finite. The $d$ components containing $u_1,\ldots,u_d$ are pairwise distinct and are not identified through infinity.

The component containing $u$ may be finite or belong to the wired infinite class; in either case, merging one of these distinct finite child components into it decreases the cluster count by one. For $A\subseteq[d]$, set $e_i$ to be open exactly when $i\in A$. Relative to the configuration in which all $e_i$ are closed, opening the edges indexed by $A$ merges $|A|$ distinct child components with the component containing $u$. Thus the number of clusters decreases by exactly $|A|$. Consequently, for a factor $C(\eta)>0$ independent of $A$, the conditional weight of the configuration $A$ is
\[
    W_\eta(A)
    =
    C(\eta)
    p^{|A|}
    (1-p)^{d-|A|}
    q^{-|A|}
    =
    C(\eta)
    \prod_{i=1}^d
    \left(\frac pq\right)^{\mathbf 1_{\{i\in A\}}}
    (1-p)^{\mathbf 1_{\{i\notin A\}}}.
\]
Summing over all $A\subseteq[d]$ gives the normalizing factor $C(\eta)((1-p)+p/q)^d$. Therefore
\[
    \frac{W_\eta(A)}{\sum_{A'\subseteq[d]}W_\eta(A')}
    =
    \prod_{i=1}^d
    \widehat p^{\mathbf 1_{\{i\in A\}}}
    (1-\widehat p)^{\mathbf 1_{\{i\notin A\}}},
\]
where $\widehat p =\frac{p}{p+q(1-p)}$. This proves the claimed conditional product distribution.
\end{proof}

The next lemma gives the uniform local survival estimate needed below. It says that no exterior configuration can make $\mathcal B_u$ arbitrarily unlikely.

\begin{lemma}[Uniform branch survival]
\label{lem:uniform-branch-survival}
Let $0<q<1$, $\Delta\ge3$, and $p>\pc$. There exists $\theta\in(0,1]$, depending only on $p,q,$ and $d$, such that, for every infinite-volume measure $\nu$ satisfying the wired-tree DLR condition and every vertex $u$, $\nu\left(\mathcal B_u\,\middle|\,\mathcal F_u^{\mathrm{ext}}\right)\ge\theta $ almost surely.
\end{lemma}

\begin{proof}
Choose $\varepsilon\in(0,1)$ sufficiently small so that $d(1-\varepsilon)\widehat p>1$. For $b\in\{0,\ldots,d\}$, define
\[
    \Psi_b(z)
    :=
    1-(1-\varepsilon)^b
    \left(
        1-(1-\varepsilon)\widehat p z
    \right)^{d-b}.
\]
We have $\Psi_0(0)=0$ and $\Psi_0'(0)=d(1-\varepsilon)\widehat p>1$, whereas $\Psi_b(0)=1-(1-\varepsilon)^b>0$ for $b\ge1$. Hence there exists $\delta\in(0,1)$, depending only on $p,q,$ and $d$, such that $\Psi_b(\delta)\ge\delta$, for $b=0,\ldots,d$.

Independently of $\omega\sim\nu$, assign to every oriented child edge $e$ a Bernoulli mark $I_e$ with parameter $\varepsilon$, and denote their joint law by
\[
    \overline\nu
    :=
    \nu
    \otimes
    \bigotimes_e \operatorname{Ber}(\varepsilon),
\]
where the product runs over all oriented child edges. Edges with $I_e=1$ are called \emph{probes}. Starting from $u$, run the following breadth-first exploration. At each active vertex $v$, inspect the events $\mathcal B_w$ for its probe children $w$, without inspecting the corresponding probe edges. If a probe child is blue, select the first such probe edge and stop. For every nonprobe child edge, reveal its state and declare its child active when the edge is open.

We first show that, conditional on $\mathcal B_u^c$, with probability at least $\delta$ the exploration selects a probe edge $e=\{v,w\}$ whose parent $v$ is active and whose child $w$ is blue. This is the unrevealed contact edge whose conditional open and closed probabilities will be compared below.

The star calculation in Lemma~\ref{lem:conditional-star-factorization} gives the following slightly more general fact. Fix an active vertex $v$ with children $v_1,\ldots,v_d$, write $e_i:=\{v,v_i\}$, and condition on the configuration $\eta$ off these $d$ edges. Let $S:=\{i:\mathcal B_{v_i}\text{ occurs}\}$ and suppose $|S|=b$. On $\mathcal B_v^c$, every edge $e_i$ with $i\in S$ must be closed. For $A\subseteq[d]\setminus S$, consider the star configuration in which, among the nonblue-child edges, $e_i$ is open exactly when $i\in A$. 

Its unnormalized conditional weight, given $\eta$ and $\mathcal B_v^c$, is
\[
    W_\eta(A)
    =
    C(\eta)
    (1-p)^b
    p^{|A|}
    (1-p)^{d-b-|A|}
    q^{-|A|},
\]
where $C(\eta)>0$ does not depend on $A$. Indeed, each nonblue-child open edge merges one distinct finite child component into the component of $v$, decreasing the cluster count by one. After removing the factors independent of $A$, the weight is proportional to $\left(p/q\right)^{|A|}(1-p)^{d-b-|A|}$. Normalizing over $A\subseteq[d]\setminus S$ shows that the $d-b$ nonblue-child edges are independent Bernoulli variables with parameter $\widehat p$.

Let $\widehat{\mathcal F}_v^{\mathrm{ext}}:=\mathcal F_v^{\mathrm{ext}}\vee\sigma\left(I_e:e\notin E(\mathcal T_v^{\mathrm{ext}})\right)$, and let $\mathcal Q_n(v)$ be the event that the exploration started from $v$ either selects a probe edge before generation $n$ or reaches generation $n$. We claim inductively that
\begin{equation}
\label{eq:cavity-finite-depth}
    \overline\nu\left(
        \mathcal Q_n(v)
        \,\middle|\,
        \widehat{\mathcal F}_v^{\mathrm{ext}},
        \mathcal B_v^c
    \right)
    \ge\delta.
\end{equation}

The case $n=0$ is immediate. Assume~\eqref{eq:cavity-finite-depth} for $n$, and condition on
\[
    S_v:=\{w:w\text{ is a blue child of }v\}=S,
    \qquad
    |S|=b.
\]
Order the children as $w_1,\ldots,w_d$ and process their branches in this order. Before processing $w_j$, let $\mathcal G_{j-1}$ be generated by $\widehat{\mathcal F}_v^{\mathrm{ext}}$, $\mathbf 1_{\mathcal B_v}$, $S_v$, and the probe marks, edge states, and descendant-exploration outcomes revealed while processing $w_1,\ldots,w_{j-1}$. 

In the $j$th branch, first reveal the probe mark on $e_j:=\{v,w_j\}$. If $w_j$ is blue and $e_j$ is a probe, the exploration selects $e_j$. If $w_j$ is nonblue and $e_j$ is not a probe, reveal its state; when it is open, continue with the depth-$n$ exploration rooted at $w_j$. Let $\mathcal G_j$ be the sigma-field after this processing, and let $\mathcal D_j$ be the event that the $j$th branch produces neither a selected probe edge nor a continuation to depth $n$.

Fix $j$ and consider $\{S_v=S\}\cap\bigcap_{k<j}\mathcal D_k$. If $w_j\notin S$, then no information from the descendant branch $\mathcal T_{w_j}^{\mathrm{ext}}$ has yet been exposed, apart from the event $\mathcal B_{w_j}^c$. After revealing the nonprobe mark and the state of $e_j$, set $\mathcal H_j:=\mathcal G_{j-1}\vee\sigma(I_{e_j},\omega(e_j))$. The reveal order gives the sigma-field inclusion $\mathcal H_j\subseteq\widehat{\mathcal F}_{w_j}^{\mathrm{ext}}\vee\sigma(\mathbf 1_{\mathcal B_{w_j}})$.

The induction hypothesis at $w_j$, followed by the tower property from $\widehat{\mathcal F}_{w_j}^{\mathrm{ext}}\vee\sigma(\mathbf 1_{\mathcal B_{w_j}})$ down to $\mathcal H_j$, therefore gives conditional probability at least $\delta$ for $\mathcal Q_n(w_j)$ whenever $w_j$ is nonblue and $e_j$ is open. The star calculation, which applies Lemma~\ref{lem:conditional-star-factorization} after forcing every edge to a blue child to be closed on $\mathcal B_v^c$, shows that, conditional on $\mathcal B_v^c$ and $S_v=S$, the remaining nonblue-child edges are independent Bernoulli variables with parameter $\widehat p$. This remains true after conditioning on the previously processed branches. Finally, the probe marks are independent of the random-cluster configuration and of one another. Combining these three facts gives, on $\bigcap_{k<j}\mathcal D_k$,
\[
    \overline\nu\left(
        \mathcal D_j
        \,\middle|\,
        \mathcal G_{j-1}
    \right)
    \le
    \begin{cases}
        1-\varepsilon,
        & w_j\in S,\\
        1-(1-\varepsilon)\widehat p\,\delta,
        & w_j\notin S.
    \end{cases}
\]

Iterating the tower property along $(\mathcal G_j)_{j=0}^d$ gives
\[
    \overline\nu\left(
        \mathcal Q_{n+1}(v)^c
        \,\middle|\,
        \widehat{\mathcal F}_v^{\mathrm{ext}},
        \mathcal B_v^c,
        S_v=S
    \right)
    \le
    (1-\varepsilon)^b
    \left(
        1-(1-\varepsilon)\widehat p\,\delta
    \right)^{d-b}
    =
    1-\Psi_b(\delta).
\]
The factor $(1-\varepsilon)^b$ corresponds to having no probe among the blue children, while the second factor corresponds to having no successful continuation through the nonblue children. Thus, for every $S$,
\[
    \overline\nu\left(
        \mathcal Q_{n+1}(v)
        \,\middle|\,
        \widehat{\mathcal F}_v^{\mathrm{ext}},
        \mathcal B_v^c,
        S_v=S
    \right)
    \ge
    \Psi_{|S|}(\delta)
    \ge
    \delta.
\]
Averaging over $S_v$ proves~\eqref{eq:cavity-finite-depth} for $n+1$.

Here generations are measured by graph distance from the starting vertex. The events $\mathcal Q_n(u)$ decrease with $n$. If $\bigcap_{n\ge0}\mathcal Q_n(u)$ occurs without selecting an edge, then the finitely branching tree of active vertices contains vertices at every depth. K\"onig's infinity lemma therefore gives an infinite downward ray of open nonprobe edges from $u$, which implies $\mathcal B_u$. Conditional continuity from above in~\eqref{eq:cavity-finite-depth} now shows that, on $\mathcal B_u^c$, the probability that the exploration selects a probe edge is at least $\delta$.

Let $E_u$ denote the probe edge selected by the exploration, and set $\mathcal S_u:=\{E_u\text{ is defined}\}$. For a fixed oriented edge $e=\{v,w\}$, let $\mathcal A_e:=\{E_u=e\}$. To determine whether $\mathcal A_e$ occurs, the exploration checks the revealed path from $u$ to $v$, the earlier tested branches, the probe marks, and the event $\mathcal B_w$. It never checks whether $e$ itself is open or closed. Therefore $\mathcal A_e$ is measurable from the probe marks and the configuration off $e$.

Write $\mathbf I$ for the complete family of probe marks $(I_f)$, where $f$ ranges over all oriented child edges of $\mathcal T_\infty$. For every edge $e$, the single-edge conditional open probability is either $p$ or $\widehat p$. Since $\mathbf I$ is independent of the random-cluster configuration, the same is true under $\overline\nu$, and hence
\[
    \frac{
        \overline\nu\left(
            \omega(e)=1
            \,\middle|\,
            \omega_{E(\mathcal T_\infty)\setminus\{e\}},
            \mathbf I
        \right)
    }{
        \overline\nu\left(
            \omega(e)=0
            \,\middle|\,
            \omega_{E(\mathcal T_\infty)\setminus\{e\}},
            \mathbf I
        \right)
    }
    \ge
    \frac{p}{1-p}.
\]
Since $\mathcal A_e$ and $\mathcal F_u^{\mathrm{ext}}$ are measurable off $e$, integrating the preceding inequality over the configurations off $e$ gives
\[
    \overline\nu\left(
        \mathcal A_e\cap\{\omega(e)=1\}
        \,\middle|\,
        \mathcal F_u^{\mathrm{ext}}
    \right)
    \ge
    \frac{p}{1-p}
    \overline\nu\left(
        \mathcal A_e\cap\{\omega(e)=0\}
        \,\middle|\,
        \mathcal F_u^{\mathrm{ext}}
    \right).
\]
On $\mathcal A_e\cap\{\omega(e)=1\}$, the active open path from $u$ reaches the parent endpoint of $e$, and the child endpoint is blue; hence $\mathcal B_u$ occurs. On $\mathcal B_u^c\cap\mathcal S_u$, the selected edge is necessarily closed. The tree has countably many edges, so the edgewise conditional inequalities may be taken to hold on one common full-measure event. Since the events $(\mathcal A_e)_e$ are disjoint, applying the inequality first to finite collections of possible selected edges and then using conditional monotone convergence gives
\[
\begin{aligned}
    \nu\left(
        \mathcal B_u
        \,\middle|\,
        \mathcal F_u^{\mathrm{ext}}
    \right)
    &\ge
    \frac{p}{1-p}
    \overline\nu\left(
        \mathcal B_u^c\cap\mathcal S_u
        \,\middle|\,
        \mathcal F_u^{\mathrm{ext}}
    \right).
\end{aligned}
\]
Let $a_u:=\nu\left(\mathcal B_u\,\middle|\,\mathcal F_u^{\mathrm{ext}}\right)$. By the preceding exploration estimate, $\overline\nu\left(\mathcal B_u^c\cap\mathcal S_u\,\middle|\,\mathcal F_u^{\mathrm{ext}}\right)\ge\delta(1-a_u)$. Therefore
\[
    a_u
    \ge
    \frac{p}{1-p}\delta(1-a_u)
    \implies
    a_u
    \ge
    \frac{p\delta}{1-p+p\delta}
    =:
    \theta
    >0.
\]
The constant $\theta$ depends only on $p,q,$ and $d$ through $\delta$.
\end{proof}

The preceding lemma gives precisely the conditional estimate needed for stochastic domination on the boundary of a finite ball. The exterior branches attached to distinct boundary vertices are edge-disjoint, so the branch-survival indicators can be exposed sequentially. At each step, Lemma~\ref{lem:uniform-branch-survival} supplies the same conditional lower bound, irrespective of the configurations already exposed in the other branches.

We restate the lemma before giving its proof.

\inducedThetaWired*

\begin{proof}
Enumerate the boundary vertices as $\partial\dreg=\{x_1,\ldots,x_m\}$. For $i\in[m]$, define $X_i:=\mathbf 1_{\mathcal B_{x_i}}$. By Lemma~\ref{lem:uniform-branch-survival}, $\nu\left(X_i=1\,\middle|\,\mathcal F_{x_i}^{\mathrm{ext}}\right)\ge\theta>0$ almost surely.

For every $j<i$, the event $\mathcal B_{x_j}$ depends only on the exterior branch rooted at $x_j$. Since the exterior branches rooted at distinct vertices of $\partial\dreg$ are edge-disjoint, $X_1,\ldots,X_{i-1}$ are measurable with respect to $\mathcal F_{x_i}^{\mathrm{ext}}$. Hence the tower property gives $\nu\left(X_i=1\,\middle|\,X_1,\ldots,X_{i-1}\right)\ge\theta$ almost surely.

The standard sequential coupling therefore produces independent Bernoulli random variables $Y_1,\ldots,Y_m$ with parameter $\theta$ such that $Y_i\le X_i$ for every $i\in[m]$ almost surely. Equivalently, the random set $A:=\{x_i:Y_i=1\}$ is an independent Bernoulli subset of $\partial\dreg$ with density $\theta$ and satisfies $A\subseteq\{x_i:\mathcal B_{x_i}\text{ occurs}\}$.

Whenever $\mathcal B_{x_i}$ occurs, $x_i$ is connected to infinity through its exterior branch. Since all infinite exterior components are identified in the wired-tree DLR boundary condition, $\{x_i:\mathcal B_{x_i}\text{ occurs}\}\subseteq\mathcal C_1\bigl(\xi_h(\sigma)\bigr)$. Consequently, $A\subseteq\mathcal C_1\bigl(\xi_h(\sigma)\bigr)$ almost surely. Thus the distribution of $\xi_h(\sigma)$ is $\theta$-wired.
\end{proof}

Lemma~\ref{lem:induced-theta-wired} supplies the missing bridge between the infinite-volume wired DLR condition and Theorem~\ref{thm:wsm}. In particular, the boundary condition induced by any wired Gibbs measure satisfies the same positive-density hypothesis, with constants depending only on $p,q,$ and $d$ and not on the measure or on the radius $h$.

\subsection{The subcritical and critical regimes \texorpdfstring{$p\le\pc$}{p <= pc}}

The result in this regime is classical; see~\cite[Theorem~10.82(d)]{grimmett2006randomcluster} and its $d$-ary analogue. We include the following direct proof in our DLR convention. No contraction argument is needed: the one-edge DLR rule and subcritical Bernoulli domination rule out infinite exterior connections; once those connections disappear, every edge has the same constant conditional open probability.

\begin{theorem}[Uniqueness at and below criticality]
\label{thm:subcritical-uniqueness}
Let $0<q<1$, $\Delta\ge3$, and $p\le\pc$. There exists a unique infinite-volume random-cluster measure satisfying the wired DLR specification of Definition~\ref{def:dlr}, given by
\[
    \mu^*
    =
    \bigotimes_{e\in E(\mathcal T_\infty)}
    \operatorname{Ber}(\widehat p).
\]
\end{theorem}

\begin{proof}
Write $d:=\Delta-1$. Since $d\widehat p\le1$, Bernoulli bond percolation with parameter $\widehat p$ has no infinite open cluster. Consequently, under the product measure $\bigotimes_e\operatorname{Ber}(\widehat p)$, the exterior-induced boundary condition is almost surely free, and the finite-tree random-cluster kernel is again product Bernoulli with parameter $\widehat p$. Thus this product measure satisfies the wired DLR specification.

Let $\nu$ be any wired DLR measure. For an edge $e$, let $\mathcal F_{e^c}:=\sigma\bigl(\omega(f):f\in E(\mathcal T_\infty)\setminus\{e\}\bigr)$. The single-edge DLR rule gives conditional open probability either $p$ or $\widehat p$, according to whether the endpoints of $e$ are already connected through the exterior configuration. Since $0<q<1$, we have $p\le\widehat p$, and hence
\[
    \nu\left(
        \omega(e)=1
        \,\middle|\,
        \mathcal F_{e^c}
    \right)
    \le
    \widehat p
    \qquad
    \nu\text{-a.s.}
\]
for every edge $e$. For any finite edge set, condition on its complement and reveal its edges sequentially. By the tower property, each conditional open probability remains at most $\widehat p$, so the usual uniform coupling gives domination by independent $\operatorname{Ber}(\widehat p)$ variables. Hence $\nu$ is stochastically dominated by Bernoulli bond percolation with parameter $\widehat p$. The condition $p\le\pc$ is equivalent to $d\widehat p\le1$. Bernoulli bond percolation with parameter $\widehat p$ therefore has no infinite open cluster on $\mathcal T_\infty$, and the same is true under $\nu$. 

Fix an edge $e=\{u,v\}$. Since $\mathcal T_\infty$ is a tree, $u$ and $v$ have no finite alternative path outside $e$. They could be identified through the wired boundary at infinity only if the exterior configuration contained an infinite open component. This has $\nu$-probability zero, so the single-edge DLR rule gives
\[
    \nu\left(
        \omega(e)=1
        \,\middle|\,
        \mathcal F_{e^c}
    \right)
    =
    \widehat p
    \qquad
    \nu\text{-a.s.}
\]
for every edge $e$. Iterating these constant one-edge conditional probabilities over each finite edge set yields
\[
    \nu
    =
    \bigotimes_{e\in E(\mathcal T_\infty)}
    \operatorname{Ber}(\widehat p).
\]
Thus every wired DLR measure equals this product measure, proving uniqueness.
\end{proof}

\begin{corollary}[Wired DLR phase diagram]
\label{cor:wired-dlr-phase-diagram}
Let $0<q<1$, $\Delta\ge3$, and $p\in(0,1)$. The wired DLR specification on $\mathcal T_\infty$ admits a unique infinite-volume random-cluster measure. If $p\le\pc$, this measure is $\bigotimes_{e\in E(\mathcal T_\infty)}\operatorname{Ber}(\widehat p)$. If $p>\pc$, it is the all-wired finite-volume limit $\mu^*$ and satisfies $\mu^*(\rho\leftrightarrow\infty)>0$.
\end{corollary}

\begin{proof}
Combine Theorems~\ref{thm:uniqueness} and~\ref{thm:subcritical-uniqueness} with Proposition~\ref{prop:wired-limit-percolates}.
\end{proof}

\begin{remark}
The division into the two regimes above agrees with the fixed-point classification in Theorem~\ref{thm:fixed-point-structure}. When $p\le\pc$, the scalar message recursion has only the trivial fixed point $x=1$. Correspondingly, no connection to the wired boundary survives at infinity, and the unique wired DLR measure reduces to independent Bernoulli percolation with parameter $\widehat p$. When $p>\pc$, the recursion has a unique nontrivial attracting fixed point $x^*>1$, and the all-wired finite-volume messages converge to $x^*$. The resulting boundary law determines the unique wired infinite-volume measure $\mu^*$, while Proposition~\ref{prop:wired-limit-percolates} separately verifies that $\mu^*(\rho\leftrightarrow\infty)>0$.

This agreement should not be interpreted as a one-to-one correspondence between all fixed points of the scalar recursion and wired DLR measures. In the supercritical regime, the trivial fixed point $x=1$ remains present but is unstable and is not selected by the wired finite-volume approximation. Thus the infinite-volume uniqueness result reflects not only the fixed points of the message recursion, but also their stability and the boundary condition used to select the limiting measure.
\end{remark}

\section{Negative dependence across wired branches}
\label{sec:negative-dependence}

Throughout this section we consider the random-cluster model with cluster weight $0<q<1$; the general network statements allow arbitrary positive edge activities. The application is to a finite tree $T$ whose leaves are all in one wired boundary class. After these leaves are identified, we denote the resulting single vertex by $\partial_{\mathrm w}$. Thus $\partial T$ is the set of leaves before wiring, whereas $\partial_{\mathrm w}$ is one vertex in the wired quotient graph.

Fix an internal vertex $x$ of $T$. Each component of $T\setminus\{x\}$, together with its incident edge to $x$ and the wired vertex $\partial_{\mathrm w}$, forms a branch joining $x$ to $\partial_{\mathrm w}$. Distinct branches have disjoint edge sets and meet only at these two vertices. The original graph is a tree, but the graph obtained after wiring the leaves need not be a tree; we therefore regard each branch as a finite network with two distinguished vertices, called its \emph{terminals}.

We first analyze the terminal-connectivity indicators of branches placed in parallel and show that their finite-dimensional joint law has the CNA+ property. We then compare the random-cluster laws within a branch conditioned on terminal connection and disconnection. This allows us to use the indicator-level result to prove a covariance inequality for increasing observables supported in disjoint collections of branches and, finally, to pass that inequality to the wired Gibbs measure. The conclusion is a branch-separated form of negative association, rather than full negative association on arbitrary disjoint edge sets.

\subsection{Connectivity indicators for parallel branches}

We begin in a slightly more general setting that isolates the interaction created by two common terminals. Write $[m]:=\{1,\ldots,m\}$. For each $i\in[m]$, let $H_i=(V_i,E_i)$ be a finite graph with two distinguished, distinct vertices $x_i,b_i\in V_i$ called its terminals. The vertices $x_i$ and $b_i$ both belong to the branch $H_i$; they need not be adjacent, and a path between them may run through any number of internal vertices and edges of $H_i$. Whether they are connected by \emph{open} edges depends on the random-cluster configuration.

Assume that the graphs $H_1,\ldots,H_m$ are pairwise vertex-disjoint before gluing. Their \emph{parallel union} $H$ is obtained from their disjoint union by the identifications
\[
    x_1=\cdots=x_m=:x,
    \qquad
    b_1=\cdots=b_m=: \partial_{\mathrm w}.
\]
Thus $x_i$ and $b_i$ are temporary copies, before gluing, of the two common vertices $x$ and $\partial_{\mathrm w}$. After gluing, distinct branches have disjoint edge sets and disjoint internal vertices, and meet only at $x$ and $\partial_{\mathrm w}$. In this abstract statement, $\partial_{\mathrm w}$ is simply the second common terminal; the notation anticipates its role as the wired boundary vertex in the tree application.

We equip $H$ with the random-cluster measure and the prescribed positive edge activities. For each branch, define
\[
    C_i
    :=
    \mathbf 1
    \left\{
        \text{there is an open path from }x\text{ to }\partial_{\mathrm w}
        \text{ using only edges of }H_i
    \right\}.
\]
Equivalently, before the identifications are made, $C_i$ records whether $x_i$ is connected to $b_i$ inside $H_i$. Let $Q_i^0$ and $Q_i^1$ be the restricted partition functions of the isolated branch $H_i$ according to whether $x_i$ and $b_i$ are disconnected or connected, respectively. These partition functions use the full random-cluster weights on $H_i$, with the component count taken in the isolated branch before the terminal copies are identified across branches. If $Q_i^0=0$, then only the connected state is possible and $C_i=1$ deterministically. If $Q_i^1=0$, then only the disconnected state is possible and $C_i=0$ deterministically. Such nonrandom coordinates may be omitted. For every remaining branch, $Q_i^0,Q_i^1>0$, and we set $\rho_i:=Q_i^1/Q_i^0.$

The joint law of $C=(C_1,\ldots,C_m)$ is an external-field tilt of an exchangeable law whose normalized rank sequence is log-concave; Pemantle's exchangeable criterion will therefore give the stronger CNA+ conclusion below.

\begin{proposition}[CNA+ for parallel branch connectivity]
\label{prop:parallel-branch-cna}
Let $0<q<1$. Under the random-cluster measure on the parallel union of $H_1,\ldots,H_m$, the vector $C=(C_1,\ldots,C_m)$ has the CNA+ property.
\end{proposition}

\begin{proof}
For a probability measure $\nu$ on $\{0,1\}^m$ and a coordinate vector $C=(C_1,\ldots,C_m)$, its \emph{rank sequence} is $R_k:=\nu\bigl(|C|=k\bigr)$ for $0\le k\le m$. If $\nu$ is exchangeable, then every vector $c\in\{0,1\}^m$ with $|c|=k$ has probability $R_k/\binom{m}{k}$. The rank sequence is called \emph{ultra-log-concave} (ULC) when the normalized sequence $\left(R_k/\binom{m}{k}\right)_{k=0}^m$ has no internal zeros and is log-concave. Here, having no internal zeros means that the indices $k$ for which $R_k>0$ form an interval, while log-concavity means
\[
    \left(
        \frac{R_k}{\binom{m}{k}}
    \right)^2
    \ge
    \frac{R_{k-1}}{\binom{m}{k-1}}
    \frac{R_{k+1}}{\binom{m}{k+1}},
    \qquad 1\le k\le m-1.
\]

Pemantle~\cite[Theorem~2.7]{pemantle2000towards} proves that, for an exchangeable binary measure, the rank sequence is ULC if and only if the measure is CNA+. In that theorem, CNA means negative association after every positive-probability coordinate conditioning, while the $+$ requires this property after arbitrary external-field tilts and projections. Thus its conclusion implies the CNA+ convention used here. We use only the implication from ULC to CNA+.

Condition on the terminal-connectivity state of every branch. Before the terminal copies are identified, the restricted partition functions factor over the branches. The component-count correction created by the two identifications depends only on the number $|C|:=\sum_{i=1}^m C_i$ of connected branches.

Before gluing, each disconnected branch contributes two terminal-containing components, while each connected branch contributes one, giving $2m-|C|$ components in total. If $|C|=0$, the two identifications leave two terminal-containing components and reduce the component count by $2m-2$. If $|C|\ge1$, at least one branch connects the common terminals, so all terminal-containing components merge into one and the reduction is $2m-|C|-1$. Thus the correction relative to the case $|C|=0$ is therefore $q^{|C|-1}$. Accordingly, define
\[
    g(k)
    :=
    \begin{cases}
        1, & k=0,\\
        q^{k-1}, & k\ge1.
    \end{cases}
\]
Consequently,
\begin{equation}
\label{eq:parallel-branch-law}
    \mathbb P(C=c)
    \propto
    \left(
        \prod_{i=1}^m \rho_i^{c_i}
    \right)
    g(|c|),
    \qquad
    c\in\{0,1\}^m.
\end{equation}

The case $m=1$ is immediate, so assume $m\ge2$. First set $\rho_i=1$ for every $i$, and denote the resulting exchangeable law by $\nu$. If $(R_k)_{k=0}^m$ is its rank sequence, then~\eqref{eq:parallel-branch-law} gives
\[
    \frac{R_k}{\binom{m}{k}}
    \propto
    g(k),
    \qquad
    (g(k))_{k=0}^m
    =
    (1,1,q,q^2,\ldots,q^{m-1}).
\]
Since $0<q<1$, the sequence $(g(k))_{k=0}^m$ is positive and log-concave. Indeed, $g(1)^2=1>q=g(0)g(2)$, whereas $g(k)^2=g(k-1)g(k+1)$ for $2\le k\le m-1$. Thus $\nu$ is exchangeable and strictly positive, and its normalized rank sequence is log-concave with no internal zeros. Hence $(R_k)_{k=0}^m$ is ULC, and the exchangeable criterion recalled above implies that $\nu$ is CNA+.

For general positive $\rho_1,\ldots,\rho_m$, the factor $\prod_i\rho_i^{c_i}$ in~\eqref{eq:parallel-branch-law} is a positive external-field tilt of $\nu$. The CNA+ property is preserved under such tilts, so the general law is also CNA+. Adjoining deterministic coordinates preserves CNA+ and does not affect this conclusion.
\end{proof}

\begin{remark}
For $m\ge3$, the exchangeable base measure $\nu$ in the proof is not strongly Rayleigh. Indeed, up to a positive normalizing constant, its rank generating polynomial is
\[
    P(z)
    =
    \sum_{k=0}^m
    \binom{m}{k}g(k)z^k
    =
    1+\frac{(1+qz)^m-1}{q}.
\]
Its zeros satisfy
\[
    1+qz
    =
    (1-q)^{1/m}
    \exp\left(\frac{2\pi i j}{m}\right),
    \qquad
    0\le j\le m-1,
\]
and hence some are nonreal when $m\ge3$. By the exchangeable strong-Rayleigh criterion~\cite[Theorem~3.8]{borceaBrandenLiggett2009}, which follows from the Grace--Walsh--Szeg\H{o} coincidence theorem, $\nu$ is therefore not strongly Rayleigh. Since positive external-field tilts are invertible rescalings of the variables, the same conclusion holds for the law in~\eqref{eq:parallel-branch-law} when all coordinates are nondegenerate. Thus the use of Pemantle's exchangeable ULC criterion genuinely goes beyond the strong-Rayleigh framework.
\end{remark}

\subsection{Conditioning on terminal connectivity}

To lift Proposition~\ref{prop:parallel-branch-cna} from connectivity indicators to arbitrary increasing branch observables, we use the series-parallel class arising from the wired-tree geometry. A \emph{two-terminal network} is a finite graph $H$ with two distinguished vertices, denoted by $s$ and $t$.

Let $H_1$ and $H_2$ be vertex-disjoint two-terminal networks with terminals $(s_1,t_1)$ and $(s_2,t_2)$, respectively. Their \emph{series composition}, denoted by $H_1\circ H_2$, is obtained by identifying $ t_1=s_2$, and taking $s_1$ and $t_2$ as the two terminals of the resulting network. Thus a terminal path must pass through $H_1$ and then through $H_2$.

Their \emph{parallel composition}, denoted by $H_1\parallel H_2$, is obtained by identifying $s_1=s_2=:s$ and $t_1=t_2=:t$. The two subnetworks then give alternative routes between the common terminals $s$ and $t$.

A \emph{two-terminal series-parallel network} is obtained from a single edge by finitely many series and parallel compositions. We also allow a finite graph to be attached to the series-parallel core through a single articulation vertex. Such an attachment does not create an additional path between the two terminals and therefore does not change the terminal-connectivity structure.

To see how the wired-tree branches fit this definition, fix an internal vertex $x$ of the finite tree and one component of $T\setminus\{x\}$. Let $y$ be the unique vertex of this component adjacent to $x$, and identify all leaves in the component with the wired vertex $\partial_{\mathrm w}$. The edge $\{x,y\}$ is placed in series with the network below $y$. Recursively, at each vertex $v$, the edge from $v$ to each child $u$ is placed in series with the network below $u$, and the resulting child networks are placed in parallel between $v$ and $\partial_{\mathrm w}$. 

Starting from the edge of the branch incident to $x$ and recursively following this decomposition toward $\partial_{\mathrm w}$ shows that the entire branch is a two-terminal series-parallel network with terminals $x$ and $\partial_{\mathrm w}$. The allowance for single-vertex attachments is a harmless generalization that covers additional finite decorations without altering this terminal-connectivity recursion.

For probability measures $\nu_1$ and $\nu_2$ on the same edge configuration space, we write $\nu_1\succeq\nu_2$ when $\nu_1$ stochastically dominates $\nu_2$, meaning that $\mathbb E_{\nu_1}[F]\ge\mathbb E_{\nu_2}[F]$ for every bounded increasing function $F$.

\begin{lemma}
\label{lem:terminal-conditioning-order}
Let $H$ be a finite two-terminal series-parallel network with terminals $s$ and $t$. Finite graphs attached to the series-parallel core through a single vertex are also allowed. Let $\mu_H^1$ and $\mu_H^0$ denote the random-cluster laws on $H$ conditioned on $s\leftrightarrow t$ and $s\nleftrightarrow t$, respectively. If both conditioning events have positive probability, then $\mu_H^1\succeq\mu_H^0$. This holds for every $q>0$ and for arbitrary positive edge activities.
\end{lemma}

\begin{proof}
For $q\ge1$, the result follows directly from FKG because $\{s\leftrightarrow t\}$ is increasing. For $0<q<1$, we argue by induction over the series-parallel construction. For a single edge, the connected law is the point mass on the open state and the disconnected law is the point mass on the closed state.

Suppose first that $H=H_1\circ\cdots\circ H_\ell$ is a series composition. Let $C_i$ indicate terminal connection in $H_i$. The global terminals of $H$ are connected exactly when $C_1=\cdots=C_\ell=1$. 
For either a series or a parallel composition and any fixed vector $c=(c_1,\ldots,c_\ell)$, the component count of the glued configuration equals the sum of the component counts in the subnetworks plus a correction depending only on $c$. Since the edge weights factor over the subnetworks, conditional on $(C_1,\ldots,C_\ell)=c$ their configurations are independent, with law $\bigotimes_i\mu_{H_i}^{c_i}$.

Thus global terminal connection corresponds to the single state $(C_1,\ldots,C_\ell)=(1,\ldots,1)$, whereas global terminal disconnection is a mixture over all other states. By the induction hypothesis, we obtain
\[
    \bigotimes_{i=1}^\ell\mu_{H_i}^1
    \succeq
    \bigotimes_{i=1}^\ell\mu_{H_i}^{c_i}
\]
for every $c\in\{0,1\}^\ell$ with $c\neq(1,\ldots,1)$. Indeed, for an increasing function on the product configuration space, replacing the factors one at a time by stochastically larger measures can only increase its expectation. It therefore dominates their mixture, which is the law conditioned on global terminal disconnection.

Suppose next that $H=H_1\parallel\cdots\parallel H_\ell$ is a parallel composition. Global terminal disconnection is equivalent to $(C_1,\ldots,C_\ell)=(0,\ldots,0)$, whereas global connection is equivalent to the complementary event. For every nonzero state $c$, the induction hypothesis gives
\[
    \bigotimes_{i=1}^\ell\mu_{H_i}^{c_i}
    \succeq
    \bigotimes_{i=1}^\ell\mu_{H_i}^{0}.
\]
By the conditional factorization above, the connected law is a mixture of the product laws on the left, while the disconnected law is the product law on the right. This proves the assertion for the parallel composition.

Finally, if a graph $K$ is attached to the series-parallel core $H'$ through one vertex, then
\[
    \kappa_{H'\vee K}(\omega',\omega_K)
    =
    \kappa_{H'}(\omega')+\kappa_K(\omega_K)-1.
\]
Hence the random-cluster weight factorizes up to a constant independent of the terminal-connectivity state of $H'$. The attachment can therefore be coupled identically under the two conditional laws. This completes the induction.
\end{proof}

\subsection{Negative association between tree branches}

We now return to the tree setting described at the beginning of the section. Let $T$ be a finite tree with all leaves in one wired boundary class, and let $\mu_T^{\mathbf 1}$ be its random-cluster measure with this all-wired boundary condition. Fix an internal vertex $x$. After the wired leaves are identified as the vertex $\partial_{\mathrm w}$, the components incident to $x$ become two-terminal series-parallel branches $H_1,\ldots,H_m$ in parallel between $x$ and $\partial_{\mathrm w}$.

\begin{theorem}[Branch-separated negative association]
\label{thm:branch-separated-na}
Let $0<q<1$, and let $I,J\subseteq[m]$ be disjoint. Suppose that
\[
    F
    =
    F\left(
        \omega_{\bigcup_{i\in I}E(H_i)}
    \right),
    \qquad
    G
    =
    G\left(
        \omega_{\bigcup_{j\in J}E(H_j)}
    \right)
\]
are bounded increasing functions. Then $\operatorname{Cov}_{\mu_T^{\mathbf 1}}(F,G)\le0$.
\end{theorem}

\begin{proof}
For each branch, let $C_i:=\mathbf 1\left\{H_i\text{ connects }x\text{ to }\partial_{\mathrm w}\right\}$. Since each branch contains a path between its terminals and all edge activities are positive, both states $C_i=0$ and $C_i=1$ have positive probability. Conditionally on $C=(C_1,\ldots,C_m)$, the branch configurations are independent, and the conditional law in $H_i$ is $\mu_{H_i}^{C_i}$. It follows that $\operatorname{Cov}(F,G\mid C)=0$.

By the same conditional product factorization, $\mathbb E[F\mid C]$ depends only on $C_I$, while $\mathbb E[G\mid C]$ depends only on $C_J$. Define these functions by $\overline F(C_I):=\mathbb E[F\mid C]$ and $\overline G(C_J):=\mathbb E[G\mid C]$. By Lemma~\ref{lem:terminal-conditioning-order}, $\overline F$ is increasing in $C_I$, and $\overline G$ is increasing in $C_J$. Indeed, changing $C_i$ from $0$ to $1$ replaces $\mu_{H_i}^0$ by the stochastically larger law $\mu_{H_i}^1$, so the expectation of an increasing observable cannot decrease. Proposition~\ref{prop:parallel-branch-cna} says that $C$ is CNA+, and hence negatively associated. Since $I\cap J=\varnothing$, $\operatorname{Cov}\bigl(\overline F(C_I),\overline G(C_J)\bigr)\le0$.
The law of total covariance now yields
\[
    \operatorname{Cov}(F,G)
    =
    \mathbb E
    \left[
        \operatorname{Cov}(F,G\mid C)
    \right]
    +
    \operatorname{Cov}
    \left(
        \mathbb E[F\mid C],
        \mathbb E[G\mid C]
    \right)
    \le0.
\]
\end{proof}

Uniqueness identifies the local limit of all-wired balls with the wired Gibbs state. It therefore lets us pass the finite branch-separated covariance inequality to the infinite-volume measure.

For a vertex $x\in\mathcal T_\infty$ and a neighbor $y$ of $x$, let $\mathcal E_x(y)$ consist of the edge $\{x,y\}$ together with all edges in the component of $\mathcal T_\infty\setminus\{x\}$ containing $y$. We call the sets $\mathcal E_x(y)$ the edge branches incident to $x$. 

\begin{corollary}[Branch-separated NA in the wired Gibbs state]
\label{cor:infinite-branch-na}
Let $0<q<1$, $\Delta\ge3$, and $p\in(0,1)$, and let $\mu^*$ be
the unique Gibbs measure for the wired DLR specification on
$\mathcal T_\infty$. Let $F$ and $G$ be bounded increasing cylinder functions. Suppose that there exist a vertex $x\in\mathcal T_\infty$ and disjoint sets $I,J$ of neighbors of $x$ such that
\[
    \operatorname{supp}(F)
    \subseteq
    \bigcup_{y\in I}\mathcal E_x(y),
    \qquad
    \operatorname{supp}(G)
    \subseteq
    \bigcup_{y\in J}\mathcal E_x(y).
\]
Then $\operatorname{Cov}_{\mu^*}(F,G)\le0$. Moreover, if $p\le\pc$, then $\operatorname{Cov}_{\mu^*}(F,G)=0$.
\end{corollary}

\begin{proof}
Since the infinite $\Delta$-regular tree and the wired specification are invariant under graph automorphisms, and since the wired Gibbs measure is unique, $\mu^*$ is automorphism-invariant. We may therefore relabel the tree so that $x$ is its root. We split the proof into the supercritical and the subcritical--critical regimes.

Suppose first that $p>\pc$. For all sufficiently large $h$, let $\dreg$ denote the ball of radius $h$ centered at $x$. The supports of $F$ and $G$ are contained in $\dreg$. Every component incident to $x$ inside $\dreg$ contains wired leaves, and hence Theorem~\ref{thm:branch-separated-na} gives 
\[
    \operatorname{Cov}_{\mu_{\dreg}^{\mathbf 1}}(F,G)
    \le0.
\]
Local convergence $\mu_{\dreg}^{\mathbf 1}\to\mu^*$ as $h\to\infty$ and boundedness of the cylinder functions allow us to pass to the limit, proving the assertion when $p>\pc$.

If $p\le\pc$, Theorem~\ref{thm:subcritical-uniqueness} identifies
$\mu^*$ as a Bernoulli product measure. Since $F$ and
$G$ have disjoint edge supports, they are independent, and hence their
covariance is zero.

Combining the two regimes completes the proof.
\end{proof}
\begin{remark}
Theorem~\ref{thm:branch-separated-na} does not imply full negative association for arbitrary disjoint edge sets. The argument requires a vertex whose incident branches separate the supports of the two observables. Such a vertex need not exist for general disjoint supports.
\end{remark}

\begin{acks}[Disclosure of generative AI use]
OpenAI's ChatGPT was used to assist with English-language editing, exposition, and pre-submission review of the manuscript. All mathematical ideas, arguments, and results are the author's own, and the author takes full responsibility for the contents of the manuscript.
\end{acks}

\bibliographystyle{plain}
\bibliography{bib}

\end{document}